\documentclass{aptpub}

\authornames{R. D. Foley and D. R. McDonald} 
\shorttitle{Yaglom limits} 

\usepackage{amsmath}		

\usepackage{enumitem}
\usepackage{mathtools}
\mathtoolsset{showonlyrefs,showmanualtags}
\usepackage{mathdots}
\usepackage{verbatim}
\usepackage{epigraph}
\usepackage{booktabs}
\usepackage{todonotes}

\usepackage{tikz}
\usetikzlibrary{arrows}

\usepackage[hyphens]{url}
\usepackage[backref,colorlinks,bookmarks=true]{hyperref} 

\renewcommand{\E}{\mathbb{E}}

\DeclareSymbolFont{bbold}{U}{bbold}{m}{n}
\DeclareSymbolFontAlphabet{\mathbbold}{bbold}
\DeclareMathOperator{\1}{\mathbbold{1}}
\newcommand{\indicator}[1]{\1 \event{#1}}

\renewcommand{\P}{\mathbb{P}}  

\newcommand{\NN}{\mathbb{N}}
\newcommand{\ZZ}{\mathbb{Z}}

\newcommand{\given}{ \mid  }
\newcommand{\pr}[1]{\P \event{ #1 }}
\newcommand{\prs}[2]{\P_{#1} \event{ #2 }}

\DeclarePairedDelimiter\abs{\lvert}{\rvert}

\DeclarePairedDelimiter\event{ \{ }{ \} }

\DeclarePairedDelimiter\set{ \{ }{ \} }

\newcommand{\qt}[1]{& &\quad\text{#1}}

\newcommand{\zK}{\prescript{}{\set{0}}{K}}

\newcommand{\ola}[1]{\overleftarrow{#1}}

\newcommand{\rhoi}{$\rho$-invariant }
\newcommand{\qsd}{QSD}
\newcommand{\qsds}{QSDs}

\renewcommand{\equiv}{\coloneqq}

\begin{document}
\title{Yaglom limits can depend on the starting state}

\authorone[Georgia Institute of Technology]{R. D. Foley} 

\addressone{Department of Industrial \& Systems Engineering \\
Georgia Institute of Technology \\
Atlanta, Georgia 30332-0205 \\ U.S.A.} 

\authortwo[The University of Ottawa]{D. R. McDonald} 

\addresstwo{Department of Mathematics and Statistics \\
The University of Ottawa \\
Ottawa, Ontario \\
Canada K1N 6N5 } 

\begin{abstract}
We construct a simple example, surely known to Harry Kesten, of an
$R$-transient Markov chain on a countable state space $S \cup \{
\delta \}$ where $\delta$ is absorbing.  The transition matrix $K$ on
$S$ is irreducible and strictly substochastic. We determine the Yaglom
limit, that is, the limiting conditional
behavior given non-absorption.  Each starting state $x \in S$ results in
a different Yaglom limit.  Each Yaglom limit is an $R^{-1}$-invariant
quasi-stationary distribution where $R$ is the convergence parameter of $K$.
Yaglom limits that depend on the starting state are related to a nontrivial
$R^{-1}$-Martin boundary.
\end{abstract}



\epigraph{The long run is a misleading guide to current affairs. In the long
run we are all dead. Economists set themselves too easy, too useless a task if
in tempestuous seasons they can only tell us that when the storm is past the
ocean is flat again.}{John Maynard Keynes}

\section{Introduction}
%

A gambler is pitted against an infinitely wealthy casino.  The gambler enters the casino with $x > 0$ dollars.  With each play, the gambler either wins a dollar with probability $b$ where $0 < b < 1/2$ or loses a dollar.  The gambler continues to play for as long as possible. What can be said about her fortune after many plays \emph{given that she still has at least one dollar}?

Seneta and Vere-Jones~\cite{Vere-Jones-Seneta} answered this question with the following probability distribution $\pi^*$:
\begin{alignat}{2}
\pi^*(y) &= \frac{1 - \rho}{a} y \left( \sqrt{\frac{b}{a}} \right)^{y - 1} \qt{for $y =  1,  2, \dots$} \label{eqn:SVLimit}
\end{alignat}
where $a = 1 - b$ and $\rho = 2\sqrt{ab}$.
Let $X_n$ be her fortune after $n$ plays.  Notice that her fortune alternates between being odd and even.  For $n$ large, Seneta and Vere-Jones proved that
\begin{alignat*}{2}
\prs{x}{X_n = y \given X_n \geq 1} &\approx
\begin{cases}
\frac{\pi^*(y)}{\pi^*(2 \NN)} &\text{ for $y$ even and $x+n$ is even} \\
 \frac{\pi^*(y)}{\pi^*(2 \NN - 1)} &\text{ for $y$ odd and $x+n$ is odd}
\end{cases}
\end{alignat*}
where $\NN \equiv \set{1,2, \dots}$ and $\P_x$ means that we also condition on $X_0 = x \in \NN$.
The probability $\pi^*$ assigns to the even and odd natural numbers is denoted by $\pi^*(2 \NN)$ and $\pi^*(2 \NN - 1)$, respectively.

The gambler's ruin problem and the Seneta--Vere-Jones' result are beautiful,
but the even/odd periodicity obscures our main point.  To remove this
distraction, assume the gambler starts with an even number of dollars $2x$,
and consider the Markov chain $X_0, X_2, X_4, \dots$
The transition matrix for this chain restricted to even states that are strictly positive is aperiodic, and
\begin{alignat}{2}
\lim_{n \to \infty}\prs{2x}{X_{2n} = y \given X_{2n} \geq 1} &=
\frac{\pi^*(y)}{\pi^*(2 \NN)} \qt{for $y \in 2\NN$.} \label{eqn:GamblersRuinAperiodic}
\end{alignat}

\renewcommand\thefootnote{\arabic{footnote}}
\setcounter{footnote}{0}

Notice that this limiting conditional distribution does not depend on the
starting state $2x$.  Whether this holds true in general for irreducible,
aperiodic sub-Markov chains is, or was, an open question.  There is neither a
proof\footnote{This is not exactly true; see the second paragraph of Section~4
	of~\cite{JackaRoberts}.} that if the limiting conditional distributions starting from different
states exist, they must be equal, nor an example showing that they might not
be.  This paper fills that gap.  We construct an example where every starting
state $x$ leads to a different limiting conditional behavior $\pi_x$.

More precisely, consider a sub-Markov chain $X_0, X_1, \dots$ on a countably
infinite state space $S$.  By a \emph{sub-Markov} chain, we mean that the
one-step transition matrix $K$ between states in $S$ is
substochastic; that is, $K(x,S) \coloneqq \sum_{y \in S} K(x,y) \leq 1$.
\emph{Strictly substochastic} means that there is at least one row $x$ such
that $ K(x, S) < 1$.  The missing mass can be thought of as representing a
transition to an absorbing state $\delta \not\in S$.  $K$ is \emph{irreducible}
if for any $x,y \in S$, there exists an $n = n(x,y)$ such that $K^n(x,y) > 0$
where $K^n$ is the matrix of $n$-step transition probabilities.  $K$ is
\emph{aperiodic} if $d = 1$ where $d = \gcd\set{n > 0 : K^n(x,x) > 0}$, which
does not depend on $x$ when $K$ is irreducible.  If $d > 1$, then $K$ is
periodic with period $d$.

We construct examples
where every starting state $x$ leads to a different limiting conditional distribution
$\pi_x$  even though the transition matrix is irreducible and aperiodic.  That is,
\begin{alignat}{2}
\prs{x}{X_n = y \given X_n \in S} &\to \pi_x(y) \qt{for $y \in S$,} \label{eqn:yaglom1}
\end{alignat}
but $\pi_x$ is different for every $x \in S$.

The paper is organized as follows.
The first several subsections of \autoref{sec:character} define some standard terms.
In
\autoref{sec:aperiodic-Yaglom} we define a limiting condition
distribution---called the  Yaglom limit---for
the aperiodic case.
Our definition is slightly unusual since we
explicitly allow the possible dependence upon the starting state by including the
starting state $x$ in the r.h.s.\ of \eqref{eqn:YaglomLimit}.  
We prove several
basic results that follow from the existence of a Yaglom limit in the
aperiodic case for a fixed starting state $x$.
Although Kesten's strong ratio limit property (SRLP) may not hold,
we show that
a generalized strong ratio limit property (GSRLP)
holds when a Yaglom limit exists for each starting
state---though each Yaglom limit may be different.

In \autoref{sec:aperiodic-Yaglom} we turn our attention to periodic Yaglom limits
where $K$ is periodic with period $d > 1$.  Many of our examples, e.g.,
the gambler's ruin problem, are periodic.
The definition of a periodic Yaglom limit is given in
\eqref{eqn:PeriodicYaglomLimit0} and again the possible dependence on the starting state
is explicit.  In addition, the definition includes the sequence of subsets
of the state space that the
process cycles through.
A periodic Yaglom limit
requires that $d$ different limits hold.
After proving
some basic properties that follow from the existence of a periodic Yaglom
limit,
we establish a series of results that greatly simplify the process of establishing a
periodic Yaglom limit.

Section~\ref{sec:dualwrev} describes a duality between $t$-invariant
measures and $t$-harmonic functions related to reversibility.  A more general
such duality is described in \autoref{sec:dualworev}. 
Section~\ref{sec:identities} describes an idea that allows a variety of 
useful identities to be derived.  
Section~\ref{sec:examples} contains a variety of examples including our primary
example dubbed the ``hub-and-two-spoke example.''  The hub-and-two-spoke
example provides an excellent medium for exploring the connections 
between Martin boundary theory---particularly, the $\rho$-Martin entrance
boundary theory---and Yaglom limits that may depend on the initial state.

Instead of reading \autoref{sec:character} next, we encourage the reader to
look at the definitions of a Yaglom limit \eqref{eqn:YaglomLimit} and a
periodic Yaglom limit \eqref{eqn:PeriodicYaglomLimit0} and then immediately jump to the
hub-and-two-spoke example in \autoref{sec:examples}. Theorem~\ref{mainone}
shows that the periodic Yaglom limit
of the hub-and-two-spoke example depends on the initial state. 
This and the other examples will motivate the results in
the sections initially skipped.  We refer back to the results in the skipped sections as we
need them when analyzing the examples.  



\section{Characterizing Yaglom limits}\label{sec:character}
Throughout, we assume that $K$ is an  irreducible, substochastic matrix.
Our goal in
this section is to characterize the Yaglom limits \eqref{eqn:yaglom1} as quasi-stationary distributions for both the aperiodic and periodic cases.

Let $\zeta$ be the exit time from $S$, also known as the time of absorption.  Notice that
\begin{alignat*}{2}
\prs{x}{X_n \in S} &=  K^n(x,S) = \prs{x}{\zeta > n},\\
\intertext{and, since $\delta$ is absorbing}
\prs{x}{X_n = y \given X_n \in S} &= \frac{K^n(x,y)}{ K^n(x, S)} \qt{for all $y \in S$.}
\end{alignat*}
From \emph{Scheff\'{e}'s Theorem}~\cite[Theorem~16.11]{Billingsley}, \eqref{eqn:yaglom1} implies convergence in total variation.  We frequently appeal to the corollary of {Scheff\'{e}'s Theorem} to conclude that
\begin{alignat*}{2}
\E_x[f(X_n) \given X_n \in S] &\to \pi_x f
\equiv \sum_{y \in S} \pi_x(y) f(y).
\end{alignat*}
whenever $f$ is a bounded function on $S$.

\subsection{Quasi-stationary distributions}
We think of a distribution $\pi$ as a nonnegative row vector with elements $\pi(y)$ for $y \in S$ that sum to $\pi(S) \leq 1$.  For $\pi$ to be a \emph{quasi-stationary distribution}, we need $\pi(S) = 1$, and
\begin{alignat}{2}
\prs{\pi}{X_n = y \given X_n \in S} &= \pi(y) \qt{for all $n \geq 0$} \label{eqn:quasi}
\end{alignat}
where $\P_\pi$ is the distribution of the chain when $X_0$ is given the distribution $\pi$.
Irreducibility implies that a quasi-stationary distribution for $K$ must be strictly positive.  If $K$ is also strictly substochastic, then $0 < \prs{\pi}{X_1 \in S} < 1$.
We will use the abbreviation \qsd\  for quasi-stationary distribution, and \qsds\  for the plural.

\subsection{Invariant and excessive measures and \qsds}
If we set $n = 1$ in \eqref{eqn:quasi} and multiply both sides by $t \coloneqq \prs{\pi}{X_1 \in S}$, we have
$\pi K = t \pi$
where $0 < t \leq 1$.  
If $K$ is strictly substochastic, the factor $t$ shrinks $\pi$ to account for the missing mass.  
Thus, $\pi$ is a positive left eigenvector for the eigenvalue $t$ and will be called a \emph{$t$-invariant} \qsd.  
Similarly, a measure $\sigma$ on $S$ is $t$-invariant if
$0 \leq \sigma(x) < \infty$ for all $x \in S$ and 
$\sigma K = t \sigma$.  
If 
If  $0 \leq \sigma(x) < \infty$ for all $x \in S$ and 
$\sigma K \leq t \sigma$
then $\sigma$ is $t$-excessive.
It follows from irreducibility that a non-degenerate
$t$-invariant measure with $t > 0$ must be strictly positive.  
Invariant measure means the same as $1$-invariant measure. 

We are primarily interested in $\rho$-invariant \qsds\  where $\rho = 1/R$ and
$R$ is the (common) radius of convergence of the generating functions
\begin{alignat}{2}
	G_{x,y}(z) \equiv \sum_{n \geq 0} K^n(x,y)z^n.
  \label{eqn:GenFun}
\end{alignat}
 Seneta~\cite[Theorem~6.1]{SenetaMatrices}
  has a proof that the radius of convergence
of $G_{x,y}(z)$ is the same for every $x,y$ and also that $R < \infty$.
In addition, since $K$ is substochastic,
$1 \leq R < \infty$.  $R$ is called the \emph{convergence parameter} of $K$,
and $\rho = 1/R$ the convergence norm or spectral radius.  Either   $G_{x,y}(R)
< \infty$ for all $x,y$, or $G_{x,y}(R) = \infty$ for all $x,y$.  In the
former case, $K$ is said to be $R$-transient, and in the latter case,
$R$-recurrent.  The $R$-recurrent case tends to be far more tractable
since there is at most one $\rho$-invariant measure.

$K$ may have $t$-invariant \qsds\  for $t \neq \rho$.  For example, a
consequence of Theorem~4.2 in van~Doorn and
Schrijner~\cite{vDSstationarity1995} is that the gambler's ruin problem described in the
Introduction has a unique $t$-invariant \qsd\  for every $t \in [\rho, 1)$.
  Our reason for focusing on \rhoi \qsds\ is that Proposition~\ref{prop:rhoqsd}
  implies that only  \rhoi \qsds\ can describe the limiting conditional
  behavior when the initial distribution is concentrated on a single state.  If
  the initial distribution is allowed to have an infinite support, any $t$-invariant
  \qsd\ can be a limiting conditional distribution.  
  We give examples where each
  starting state $x$ results in a different limiting conditional distribution
  chosen from an infinite family of \rhoi \qsds.



Seneta and Vere-Jones~\cite{Vere-Jones-Seneta} allow for the possibility that an
irreducible, substochastic matrix could have multiple $\rho$-invariant measures. The
$\rho$-Martin entrance boundary theory~\cite{Dynkin} describes the
cone of such measures---more on that in \autoref{boundary}.  Nonetheless,
we do not know of any earlier examples of Yaglom limits that depend on the starting state.

\subsection{Harmonic and superharmonic functions}
Analogous to $t$-invariant and $t$-excessive measures, we have $t$-harmonic and
$t$-superharmonic functions.  A real-valued function $h$ on $S$ will be $t$-superharmonic if $h
\geq 0$ and $Kh \leq th$.  
To avoid continually writing ``nonnegative,'' we have included nonnegativity as part of
the definition of $t$-superharmonic.  
If in addition, $Kh = th$, then $h$ is $t$-harmonic.
Due to irreducibility,
a non-degenerate $t$-superharmonic function with $t > 0$ must be strictly
positive.  
Harmonic and superharmonic mean the same as $1$-harmonic and $1$-superharmonic,
resp.

\subsection{The aperiodic case}\label{sec:aperiodic-Yaglom}

A proper probability distribution $\pi_x$ describes the limiting conditional behavior starting from state $x \in S$ if \eqref{eqn:yaglom1} holds, or equivalently,
\begin{alignat}{2}
\frac{K^n(x,y)}{ K^n(x, S)} &\to \pi_x(y) \qt{for all $y \in S$}. \label{eqn:YaglomLimit}
\end{alignat}
where $\pi_x$ is a proper probability distribution on $S$.
We refer to $\pi_x$ as  being the \emph{limiting conditional distribution} or as being the \emph{Yaglom limit}. 
Our use of the term Yaglom limit is slightly nonstandard since the r.h.s.\ may
depend on the starting state $x$.

\begin{lemma}\label{limitofratios}
Let $K$ be irreducible, aperiodic, and substochastic.
If $\pi_x$ is a \emph{Yaglom limit} as in \eqref{eqn:YaglomLimit}, then the following hold:
\begin{alignat}{2}
\lim_{n\to\infty}\frac{K^{n+1}(x,S)}{ K^n(x, S)} &= \rho \label{eqn:asympsurvivalprob} \\
\lim_{n\to\infty}\frac{K^{n+1}(x,y)}{ K^n(x, y)}&=\rho \label{eqn:ratio} \qt{for all $y \in S$.}
\end{alignat}
%
%
\end{lemma}

\begin{remark} 
	Recall that $\zeta$ is the exit time from $S$. 
	If the initial distribution is a $t$-invariant \qsd, then $\zeta$ is a 
	geometric random variable, and the conditional probability of remaining 
	in $S$ for one more step is always $t$.  
	Call $\prs{x}{\zeta > n + 1 \given \zeta > n}$ the survival probability 
	at age $n$ starting from state $x$.  
	Thus, \eqref{eqn:asympsurvivalprob} 
	states that the 
	(one-step) survival probability starting from state $x$ is 
	asymptotically $\rho$.  
	Consequently, if $\pi_x$ in 
	\eqref{eqn:YaglomLimit} is a $t$-invariant \qsd, then 
	\eqref{eqn:asympsurvivalprob} implies $t = \rho$.
\end{remark}

\begin{remark} 
	Kesten~\cite{KestenSubMarkov} proves \eqref{eqn:ratio} in Lemma~4 
	under a uniform
	aperiodicity assumption. He also gives some of the earlier
	history of this result.  
	Equation~\eqref{eqn:asympsurvivalprob} can be proven with an argument 
	similar to Kesten's proof of Lemma~4 
	assuming uniform
	aperiodicity and the existence of a $\rho$-excessive probability
	measure $\mu$ (that is, $\mu K \leq \rho \mu$) 
	\cite{McFoleyII}.  These results are
	quite general, and do not require a Yaglom limit to hold.  If, however,
	a Yaglom limit holds, then we have the following simple proof.
%
%
\end{remark}

\begin{proof}[Proof of Lemma~\ref{limitofratios}]
Since $K(y,S)$ is a bounded function in $y$, we can use the corollary to {Scheff\'{e}'s Theorem} to show that
\begin{alignat*}{2}
\frac{K^{n+1}(x,S)}{ K^{n}(x, S)}&=\sum_{y \in S}\frac{K^{n}(x,y)}{K^{n}(x,S)} K(y, S) \\
&\to \sum_{y \in S}\pi_x(y) K(y, S) \\
&\eqqcolon L \in (0,1].
\end{alignat*}
If we think of the l.h.s.\ of \eqref{eqn:YaglomLimit} as being of the form
$a_n/b_n$, we have just argued that $b_{n + 1}/b_n \to L$. From
\eqref{eqn:YaglomLimit}, we know that $(a_{n + 1}/a_n)(b_n/b_{n + 1}) \to 1$;
hence, the ratios $a_{n + 1}/a_n$ must also converge to $L$. To finish the
proof, we need only show that $L = \rho$.  

Since the root test is stronger than the ratio test, it follows that $a_n^{1/n}
\to L$.  
If $K$ is aperiodic, Theorem~A of~\cite{Vere-Jones} implies 
that $[K^{n}(x,y)]^{1/n} \to \rho$. 
Since $K^n(x,y) = a_n$, we know that $L = \rho$.  
%
\end{proof}

\begin{proposition}\label{prop:rhoqsd}
Let $K$ be irreducible, aperiodic, and substochastic.
If $\pi_x$ is a Yaglom limit as in \eqref{eqn:YaglomLimit}, then
$\pi_x$ is a $\rho$-invariant \qsd.
\end{proposition}

\begin{proof}
We essentially repeat the proof of Theorem~3.2 in~\cite{Vere-Jones-Seneta}.
Starting from \eqref{eqn:ratio},
\begin{alignat*}{2}
\rho &= \lim_{n\to\infty}\frac{\sum_z K^{n}(x,z)K(z,y)}{ K^{n}(x,y)}\\
&= \lim_{n\to\infty}\left(\frac{K^n(x,S)}{ K^{n}(x,y)}\sum_z \frac{ K^{n}(x,z)}{K^n(x,S)} K(z,y)\right) \\
&= \frac{1}{\pi_x(y)}\sum_z\pi_x(z)K(z,y).
\end{alignat*}
In the last step, we used \eqref{eqn:YaglomLimit} twice, and we again used the corollary of {Scheff\'{e}'s Theorem} to interchange the limit and sum since $K(z,y)$ is a bounded function of $z$.
\end{proof}

%
%
%

\subsubsection{Kesten's strong ratio limit property (SRLP).}

For a nonnegative matrix $K$---not necessarily substochastic---Kesten~\cite{KestenSubMarkov} defines the \emph{strong ratio limit property} as the existence of a strictly positive constant $R$, a strictly positive function $h$ on $S$, and a strictly positive measure $\pi$ on $S$ such that
\begin{alignat*}{2}
\lim_{n \to \infty} \frac{K^{n + m}(u,v)}{K^n(x,y)} = R^{-m}\frac{h(u) \pi(v)}{h(x) \pi(y)}  \qt{for all states $u,v,x,y$ and all $m \in \ZZ$}.
\end{alignat*}

Our example in \autoref{subsec:ModEx} shows that an irreducible, aperiodic substochastic matrix may not have Kesten's strong ratio limit property.  We are able to prove a different property that our example does possess and will prove useful.
Let us say that a nonnegative matrix $K$ has the \emph{generalized strong ratio limit
	property (GSRLP)} if there exists a strictly positive constant $R$, a strictly positive function $h$, and \rhoi \qsds\ $\pi_x$ on $S$ for every $x \in S$  such that
\begin{alignat}{2}
\lim_{n \to \infty} \frac{K^{n + m}(u,v)}{K^n(x,y)} = R^{-m}\frac{h(u) \pi_u(v)}{h(x) \pi_x(y)}  \qt{for all states $u,v,x,y$ and all $m \in \ZZ$}. \label{eqn:wrlp}
\end{alignat}
Much of the remainder of this subsection is closely related to the proof of
Kesten's Theorem~2~\cite{KestenSubMarkov}.

If $K$ is substochastic, irreducible and \eqref{eqn:YaglomLimit} holds for a fixed state $x_0$, then $K^n(y,S)/K^n(x_0,S)$ is a bounded sequence in $n$ for each $y \in S$.  To see this, choose $m$ so that $K^m(x_0,y) > 0$, and then we have
\begin{alignat}{2}
K^m(x_0,y)\frac{K^n(y,S)}{K^n(x_0,S)} &\leq \frac{K^{n+m}(x_0,S)}{K^n(x_0,S)}
\to \rho^m \qt{using \eqref{eqn:asympsurvivalprob}.}
\end{alignat}
Since $K^n(y,S)/K^n(x_0,S)$ is bounded, we can choose a convergent subsequence $\mathcal{N}(y)$ for each $y$, and define the subsequential limit
\begin{alignat}{2}
 \hat{h}(y) &\coloneqq \lim_{\substack{n \to \infty \\ n \in \mathcal{N}(y)}} \frac{ K^n(y,S)}{K^n(x_0,S)} \label{eqn:lambda} \qt{for all $y \in S$}.
 \intertext{If \eqref{eqn:lambda} holds for all subsequences, that is , if}
 \hat{h}(y) &= \lim_{n \to \infty} \frac{ K^n(y,S)}{K^n(x_0,S)} \qt{for all $y \in S$} \label{eqn:SurvivalRatio}
\intertext{then}
\frac{\hat{h}(y)}{\hat{h}(z)} &= \lim_{n \to \infty} \frac{K^n(y,S)}{K^n(z,S)}  \qt{for all $y,z \in S$}
\end{alignat}
and $\hat{h}(x_0) = 1$.
When \eqref{eqn:SurvivalRatio} holds, we will say that the Jacka-Roberts condition 
\cite[Lemma~2.3]{JackaRoberts} holds.  

\begin{proposition}\label{prop:aperiodicQSD}
Let $K$ be irreducible, aperiodic, and substochastic.
If \eqref{eqn:YaglomLimit} holds for every $x \in S$ and the Jacka-Roberts
condition \eqref{eqn:SurvivalRatio} holds,
then the generalized ratio limit property \eqref{eqn:wrlp} holds, and $\hat{h}$ is
$\rho$-superharmonic.
If in addition the support of $K(x,\cdot)$ is finite for every $x\in S$, then $\hat{h}$ is $\rho$-harmonic.
\end{proposition}
\begin{remark}
	The periodic example in \autoref{subsec:remlife} can be
	modified to show that $\hat{h}$ may not be $\rho$-harmonic if the
	finite support assumption does not hold.  
	The Jacka-Roberts condition 	
	\eqref{eqn:SurvivalRatio} plays a key role; 
	throughout we use $\hat{h}$ to
	designate this survival ratio.
\end{remark}
\begin{proof}[Proof of Proposition~\ref{prop:aperiodicQSD}]

First, we argue that $\hat{h}$ is strictly positive.  We know that $\hat{h}(x_0) = 1$.  Choose $m$ such that $K^m(x,x_0) > 0$.  Now,
\begin{alignat*}{2}
{K^{m + n}(x,S)} &= \sum_z K^m(x,z)K^n(z,S) \\
\frac{K^{m + n}(x,S)}{K^{m + n}(x_0,S)} &= \sum_z K^m(x,z) \frac{K^n(z,S)}{K^n(x_0,S)} \frac{K^n(x_0,S)}{K^{m + n}(x_0,S)} \qt{let $n \to \infty$,}\\
\hat{h}(x) &\geq \sum_z K^m(x,z) \hat{h}(z) R^m \qt{from Fatou's lemma} \\
&\geq K^m(x,x_0) \hat{h}(x_0) R^m \\
&> 0
\end{alignat*}
since we already know that $R = \rho^{-1} > 0$.

Next, to see that \eqref{eqn:wrlp} holds,
\begin{alignat*}{2}
\frac{K^{m + n}(u,v)}{K^n(x,y)} &=\frac{K^{m + n}(u,v)}{K^n(u,v)}\frac{K^n(u,v)}{K^n(u,S)}\frac{K^n(u,S)}{K^n(x,S)}
\frac{K^n(x,S)}{K^n(x,y)}\to \rho^m \frac{\pi_u(v) \hat{h}(u)}{\pi_x(y) \hat{h}(x)}
\end{alignat*}
where Proposition~\ref{prop:rhoqsd} guarantees that $\pi_x$ is a \rhoi \qsd\ for every $x \in S$.

To show that $\hat{h}$ is $\rho$-harmonic when the support of $K(x,\cdot)$ is finite, return to the argument at the beginning of this proof, and set $m = 1$.  Instead of using Fatou's lemma, use the finite support to replace the first inequality with equality, which shows that $\hat{h}$ is $\rho$-harmonic.
\end{proof}

Even if $\hat{h}$ is known to be
$\rho$-harmonic, it can still be difficult to determine $\hat{h}$ since
$K$ may have other $\rho$-harmonic functions
as our
hub-and-two-spoke example illustrates; see \eqref{eqn:harmonic}.
Often, Kesten's Theorem~1~\cite{KestenSubMarkov} can help since it gives conditions guaranteeing a unique, up to multiplicative constants, $\rho$-harmonic function.

One vexing open question is that we do not know whether the existence of a Yaglom
limit starting from a particular state $x$ implies the existence of Yaglom limits starting from
other states.  

\subsection{Periodic Yaglom limits \& \qsds}\label{sec:dotheperiodic}
Fix the starting state $x \in S$.  If the period $d >
1$, then we can partition $S$ into $d$ classes labeled $S_0, \dots, S_{d - 1}$
so that $y \in S_k$ iff $K^{n d + k}(x,y) > 0$ for $n$ sufficiently large.
Hence, $x \in S_0$, and $\event{X_{n d + k} \in S_k} = \event{X_{n d + k} \in
S}$.   For brevity, let $S_{k + j} \equiv S_{k + j \pmod d}$.  We will say that
$K$ has a \emph{periodic Yaglom limit starting from x} if for all $k \in
\set{0,\dots, d - 1}$ and all $y \in S_k$,
\begin{alignat}{2}
 \lim_{n \to \infty} \prs{x}{X_{nd + k} = y \given X_{n d + k} \in S}
 &=  \frac{\pi_x(y)}{\pi_x(S_k)}
  \label{eqn:PeriodicYaglomLimit0}
\end{alignat}
where $\pi_x$ is a probability measure on $S$ with $\pi_x(S) = 1$.
We will show $\pi_x$ must be a \rhoi \qsd.  Unlike when $K$ is stochastic, $\pi_x(S_k)$ may take values other than $1/d$.

Our starting point for the periodic case is slightly different than \eqref{eqn:yaglom1}.
Instead, we assume that there exists a $k$ such that
%
\begin{alignat}{2}
 \lim_{n \to \infty} \prs{x}{X_{nd + k} = y \given X_{n d + k} \in S}
 &= \pi^k_x(y) \qt{for all $y \in S_k$},
  \label{eqn:PeriodicYaglomLimit1}
 \intertext{or equivalently}
\lim_{n \to \infty} \frac{K^{n d + k}(x,y)}{ K^{n d + k}(x,S)}  &= \pi^k_x(y)
 \qt{for all $y \in S_k$}.
 \label{eqn:PeriodicYaglomLimit2}
\end{alignat}
where $\pi^k_x$ is a probability measure on $S$ with $\pi^k_x(S_k) = 1$.  If
\eqref{eqn:PeriodicYaglomLimit2} holds for some $k$, then
Proposition~\ref{periodicpro} shows that we have a periodic Yaglom limit
starting from $x$; that is, \eqref{eqn:PeriodicYaglomLimit0} holds for all $k$,
and $\pi_x$ is a \rhoi \qsd.

If a periodic Yaglom limit holds for some $x \in S_0$, we have not been able to prove
that a periodic Yaglom limit holds starting from some other state in $S_0$.  However, if
a periodic Yaglom limit holds for every state in $S_0$, then
Proposition~\ref{periodicproII} shows that a periodic Yaglom limit holds
starting from any state in $S$.

%

\begin{lemma}\label{periodicuseit}
Let $K$ be irreducible, substochastic, and periodic with period $d > 1$.
If \eqref{eqn:PeriodicYaglomLimit2} holds for some $k \in \set{0, \ldots, d -
	1}$, then
\begin{alignat}{2}
\lim_{n\to\infty}\frac{K^{(n+1)d+k}(x,S_k)}{ K^{nd+k}(x, S_k)} &= \rho^d \label{eqn:periodicL1}
\\
\lim_{n\to\infty}\frac{K^{(n+1)d+k}(x,y)}{ K^{nd+k}(x, y)} &= \rho^d \qt{for $y \in S_k$}.  \label{eqn:periodicL2}
\end{alignat}

\end{lemma}


\begin{proof}

Using the corollary to Scheff\'{e}s Theorem,
\begin{alignat*}{2}
\frac{K^{(n+1)d+k}(x,S)}{ K^{nd+k}(x, S)}&=\sum_{y \in S}\frac{K^{nd+k}(x,y)}{K^{nd+k}(x,S)} K^d(y, S)\\
&\to \sum_{y \in S}\pi_x^k(y) K^d(y, S) \\
&\eqqcolon L \in (0,1].
\end{alignat*}
As in the proof of Lemma~\ref{limitofratios}, 
\eqref{eqn:PeriodicYaglomLimit2} implies that for any $y \in S_k$,
\begin{alignat*}{2}
\frac{K^{(n+1)d+k}(x,S)}{ K^{nd+k}(x, S)} &\sim \frac{K^{(n+1)d+k}(x,y)}{
K^{nd+k}(x, y)} \to L.
\end{alignat*}
Since the root test is stronger\footnote{See
Clark~\cite[11.5.3 and Corollary~255]{Clark2012}
or~\cite[2.5]{ClarkSeqSer} for a discussion.}
	than the ratio test, $[K^{nd+k}(x,y)]^{1/(nd)} \to L^{1/d}$, but
Theorem~A~\cite{Vere-Jones} implies that for any $y\in S_k$, $[K^{nd+k}(x,y)]^{1/(n d + k)} \to \rho$.  Hence, $L = \rho^d$.

\end{proof}

\begin{proposition}\label{periodicpro}
Let $K$ be irreducible, substochastic, and periodic with period $d > 1$.
	Recall that $x$ is a fixed state in $S_0$.
Suppose that \eqref{eqn:PeriodicYaglomLimit2} holds for some
$k \in \set{0, 1, \dots, d- 1}$.  Then \eqref{eqn:PeriodicYaglomLimit2} holds for all
$k \in \set{0, 1, \dots, d- 1}$. Moreover, 
there is a \rhoi \qsd\ $\pi_x$ such that
$\pi^k_x(y)=\pi_x(y)/\pi_x(S_k)$ for $y\in S_k$ for each $k \in \set{0, 1, \dots, d- 1}$. Hence,
\begin{alignat*}{2}
 \lim_{n \to \infty} \prs{x}{X_{nd + k} = y \given X_{n d + k} \in S}
 &= \frac{\pi_x(y)}{\pi_x(S_k)} \qt{for all $k$ and all $y \in S_k$}.
\end{alignat*}
In addition, the $\pi_x^k(y)$ in \eqref{eqn:PeriodicYaglomLimit2} can be iteratively computed by using
\begin{alignat}{2}
\pi_x^{k + 1 \pmod d}(y) &= \frac{1}{\rho_k(x)} \sum_{z \in S_k} \pi_x^k(z) K(z,y)
\label{eqn:itpiky} 
\intertext{where the probability of surviving for one more step given the
	current distribution is $\pi_x^k$ is given by}
\rho_k(x) &=  \sum_{z \in S_k} \pi_x^k(z) K(z,S)
\label{eqn:survivalk}
\end{alignat}
Furthermore, the \rhoi \qsd\ $\pi_x$ can be constructed as
\begin{alignat}{2}
	\pi_x &= \sum_{k = 0}^{d - 1} c_k \pi^k_x \label{eqn:pixdefn}
\end{alignat}
where
  \begin{alignat}{2}
	  c_0 &= \frac{1}{1 + \frac{\rho_0(x)}{\rho} + \frac{\rho_0(x)\rho_1(x)}{\rho^2} + \dots +  \frac{\rho_0(x)\rho_1(x) \dots \rho_{d - 2}(x)}{\rho^{d - 1}}} \label{eqn:c0defn} \\
	  c_k &= c_0 \frac{\rho_0(x) \dots \rho_{k - 1}(x)}{\rho^k} \quad\text{for $k \in \set{1, 2, \dots, d - 1}$} \label{eqn:ckdefn}
\end{alignat}
\end{proposition}
\begin{remark}
Lemma~11 of Ferrari and Rolla~\cite{FerrariRolla}
	and Proposition~\ref{periodicpro}
are closely related.  Equations
\eqref{eqn:prodrho} and that $c_k \rho_k/\rho = c_{k +1}$
are in their Lemma~11.
\end{remark}
\begin{remark}\label{rem:rho}
In the following special case, $\rho_k(x)$ is easy to compute.
Let $\Delta \coloneqq \set{x \in S : K(x, S) < 1}$ be the set of possible exit states.
If $\Delta \subset S_j$ for some $j$, then
\begin{equation}
\rho_k(x) =
\begin{cases}
	\rho^d \quad\text{for $k = j$} \\
	1      \quad\text{for $k \neq j$}
\end{cases}
\end{equation}
which follows from \eqref{eqn:prodrho} below.
\end{remark}
\begin{proof}
	Recall that $\zeta$ is the exit time from $S$.  Let $s_n \equiv s_n(x) \equiv
  \prs{x}{\zeta > n + 1 \given \zeta > n}$ for $n \in \NN_0 \equiv \set{0} \cup
  \NN$.  Clearly, $s_n > 0$, and so is $\inf_n {s_n}$ since
  \eqref{eqn:periodicL1} implies that $\liminf_n s_n \geq \rho^d > 0$.

We now show that $s_{nd + k} \to \rho_k$ where $\rho_k \equiv \rho_k(x) = \sum_{z \in S_k}\pi^k_x (z) K(z,S_{k + 1})$.
\begin{alignat*}{2}
1 - \frac{\rho_k}{s_{nd + k}} &=\frac{\sum_{z \in S_{k}}K^{nd+k}(x,z) K(z,S_{k + 1})}{K^{nd+k}(x,S_{k}) s_{n d + k}  } - \frac{\sum_{z \in S_k} \pi^k_x (z) K(z,S_{k + 1}) }{s_{nd + k}} \\
&\leq \sum_{z \in S_k} \abs*{\frac{K^{n d + k} (x,z)}{K^{n d + k} (x,S_k)}  -  \pi^k_x(z) } \, \frac{K(z,S_{k + 1})} {s_{nd + k }}   \\
&\leq \sum_{z \in S_k} \abs*{\frac{K^{n d + k} (x,z)}{K^{n d + k} (x,S_k)}  -  \pi^k_x(z) } \, \frac{1}{\inf_n s_{nd + k}} \\
&\to 0.
\end{alignat*}
where we used \eqref{eqn:PeriodicYaglomLimit2}.

For brevity, write $k + 1$ for $k + 1 \pmod d$.
The next step is to establish that \eqref{eqn:PeriodicYaglomLimit2} holds when $k$ is replaced by $k+1$ and to relate $\pi^{k + 1}_x$ with $\pi^k_x$. For $y \in S_{k + 1}$,
\begin{alignat}{2}
\frac{K^{n d + k + 1}(x,y)}{K^{n d + k + 1}(x,S_{k + 1})} &= \sum_{z \in S_k} \frac{K^{n d + k}(x,z)}{K^{n d + k}(x,S_k)} K(z,y) \frac{1}{s_{n d + k}} \notag \\
&\to \sum_{z \in S_k} \pi^k_x(z) K(z,y) \frac{1}{\rho_k} \notag \\
&\eqqcolon \pi^{k + 1}_x(y). \label{eqn:pikp1}
\end{alignat}
Thus, $\pi^{k + 1}_x$ is a probability measure on $S$ with $\pi^{k + 1}_x (S_{k + 1}) = 1$, as can be seen by summing over $y \in S_{k + 1}$.  Thus, \eqref{eqn:itpiky} holds.

Now that we have established that \eqref{eqn:PeriodicYaglomLimit2} holds when
$k$ is replaced by $k + 1$, we can repeat the above arguments $d - 2$
additional times establishing that the above expressions in this proof hold for
all $k \in \set{0, \dots, d - 1}$.

Next, since
\begin{alignat}{2}
\frac{K^{(n + 1)d + k}(x, S_k)}{K^{n d + k}(x, S_k)} &= \prod_{j = 0}^{d - 1} s_{n d + k + j}, \notag \\
\rho^d &= \rho_0(x) \rho_1(x) \dots \rho_{d - 1}(x) \qt{as $n \to \infty$} \label{eqn:prodrho}
\end{alignat}
where we used \eqref{eqn:periodicL1}.

Though the notation suppresses the dependency,
$\rho_k$, $s_{nd + k}$
and the forthcoming $c_0, \dots, c_{d - 1}$ all depend on the
initial state being $x$.
Construct $c_0$, \dots, $c_{d - 1}$ as in \eqref{eqn:c0defn} and
\eqref{eqn:ckdefn}.  The constants are nonnegative, and $c_0 + \dots + c_{d -
1} = 1$.  Define $\pi_x$ as in \eqref{eqn:pixdefn}.  Since $\pi_x(S_k) = c_k$,
$\pi_x$ is a probability measure on $S$, and if $y \in S_k$, then $\pi^k_x (y)
= \pi_x(y)/\pi_x(S_k)$ for all $k$.

The last step is to show that $\pi_x$ is $\rho$-invariant.  We will use the fact that
$c_k \rho_k/\rho = c_{k + 1}$ and that
from \eqref{eqn:pikp1},
\begin{alignat*}{2}
\rho_k \pi^{k + 1}_x (y) = \sum_{z \in S_k} \pi^k_x(z) K(z,y)
\end{alignat*}
Let $y \in S$.  Since $y \in S$, there exists some $k = k(y)$ such that $y \in S_{k + 1}$.  Since
\begin{alignat*}{2}
\sum_{z \in S} \pi_x(z) K(z,y) &= \sum_{z \in S_k} c_k \pi^k_x(z) K(z,y) \\
&= c_k \rho_k \pi^{k + 1}_x (y) \\
&= \rho c_{k + 1} \pi^{k + 1}_x (y) \\
&= \rho \pi_x(y),
\end{alignat*}
$\pi_x$ is a \rhoi \qsd.
\end{proof}

If we
assume that a periodic Yaglom limit exists starting from every state in one class, say $S_0$,
then we can show that under certain conditions a periodic Yaglom limit will
exist starting from every state in $S$.  Since we will be assuming that a
periodic Yaglom limit exists starting from any state in $S_0$, we drop the
assumption that $x$ is a fixed starting state in $S_0$, which means that we need to be careful about quantities that depend on the starting state.
However, we will still need a fixed state $x_0$ in $S_0$ as a reference point for the following definition.
If the following limit exists, which is similar to \eqref{eqn:SurvivalRatio} except with $nd$ instead of $n$, then define
\begin{alignat}{2}
 \hat{h}_0(y) & \coloneqq \lim_{n \to \infty}
 \frac
 { K^{nd}(y,S_0)}
 {K^{nd}(x_0,S_0)}
 \qt{for all $y \in S_0$}. \label{eqn:SurvivalRatioPeriodic}
 \intertext{If \eqref{eqn:SurvivalRatioPeriodic} holds, then it follows that}
\frac{\hat{h}_0(y)}{\hat{h}_0(x)} &= \lim_{n \to \infty} \frac{K^{nd}(y,S_0)}{K^{nd}(x,S_0)} \notag \qt{for all $x,y \in S_0$}.
\end{alignat}
Equation~\ref{eqn:SurvivalRatioPeriodic} is the Jacka-Roberts condition
\eqref{eqn:SurvivalRatio} for the process watched on $S_0$.


In the following proposition when the starting state is in a class other than $S_0$, do not relabel the classes.
\begin{proposition}\label{periodicproII}
Let $K$ be irreducible, substochastic, and periodic with period $d > 1$.
If \eqref{eqn:PeriodicYaglomLimit1} holds for all $x \in S_0$,
\eqref{eqn:SurvivalRatioPeriodic} holds, and the support of $K(u, \cdot)$ is finite for all $u \in S$, then
there is a periodic Yaglom limit starting from any state in $S$.  More precisely, if $u \in S_j$ where $j \in \set{1, \dots, d - 1}$, then

\begin{alignat}{2}
\frac{K^{nd + k - j}(u,y)}{K^{nd + k - j}(u,S_k)} &\to
\frac{\pi_u(y)}{\pi_u(S_k)} \qt{for all $k$ and all $y \in S_k$} \label{eqn:p2}
\end{alignat}
where $\pi_u$ is a \rhoi \qsd.
In particular, for $k = 0$, the r.h.s.\ of \eqref{eqn:p2} is given by
\begin{alignat}{2}
	\pi^{-j}_u(y) &\coloneqq \sum_{x \in S_0} w_{u,x} \pi^0_x(y) \qt{for $y \in S_0$},
	\label{eqn:p3}
\\	
w_{u,x} &\coloneqq  \frac{K^{d - j}(u,x)\hat{h}_0(x)}{\sum_{z \in S_0} K^{d - j}(u,z)\hat{h}_0(z)}
\label{eqn:p35}
\end{alignat}
where the weights $w_{u,x}$ are nonnegative and sum to one.
In addition,
\begin{alignat}{2}
\frac
{K^{nd + d - j}(u,S_0)}
{K^{nd + d - j}(v,S_0)}
&\to
\frac
{\sum_{x \in S_0} K^{d - j}(u, x)  \hat{h}_0(x) }
{\sum_{z \in S_0} K^{d - j}(v, z)  \hat{h}_0(z) }
\label{eqn:p4}
\end{alignat}
\end{proposition}

\begin{remark}\label{rem:intuition}
	The terms involving $\hat{h}_0(\cdot)$ account for the relative likelihood of a long remaining lifetime, which might vary over the states in $S_0$.
	If the hypotheses of Prop.~\ref{periodicproIII} hold, then the denominator in \eqref{eqn:p35} simplifies to $\rho^{d - j} \hat{h}(u)$ where $\hat{h}(u)$
	is defined just prior to Prop.~\ref{periodicproIII}.

	The example in \autoref{subsec:remlife}
	shows that $\hat{h}$ may not be $\rho$-harmonic
	if the finite support assumption does not hold.
\end{remark}
\begin{proof}
	It suffices to establish that \eqref{eqn:p2} holds for any $j \in \set{1, \dots, d - 1}$ but only one $k$ since Prop.~\ref{periodicpro} can then be used to show that \eqref{eqn:p2} holds for all $k$ and that $\pi_u$ is a \rhoi \qsd.
We will show that it holds for $k = 0$ by letting $k = d$ in the l.h.s.\ of \eqref{eqn:p2} and showing that it converges to the r.h.s.\ of \eqref{eqn:p3}.
At which point, everything follows from Prop.~\ref{periodicpro} except \eqref{eqn:p4}.

Consider the numerator of the l.h.s.\ of \eqref{eqn:p2} with $k = d$ divided by $K^{nd}(x_0,S_0)$.  That is, for $y \in S_0$,
\begin{alignat*}{2}
\frac{K^{(n + 1)d - j}(u,y)}{K^{nd}(x_0,S_0)}
&=
{\sum_{x \in S_0} K^{d - j}(u, x)
\frac{K^{nd}(x,y)}{K^{nd}(x,S_0)}
\frac{K^{nd}(x,S_0)}{K^{nd}(x_0,S_0)} }
\\
&\to
\sum_{x \in S_0} K^{d - j}(u, x)
\pi^0_x(y)
\hat{h}_0(x)
\end{alignat*}
where we used the finite support of $K^{d - j}(u, \cdot)$ to justify interchanging the limit and summation and then used \eqref{eqn:SurvivalRatioPeriodic} and \eqref{eqn:PeriodicYaglomLimit1}.

Now, consider the denominator of the l.h.s.\ of \eqref{eqn:p2} also with $k = d$ and also divided by $K^{nd}(x_0,S_0)$.  That is,
\begin{alignat*}{2}
\frac{K^{(n + 1)d - j}(u,S_0)}{K^{nd}(x_0,S_0)}
&=
{\sum_{x \in S_0} K^{d - j}(u, x)
\frac{K^{nd}(x,S_0)}{K^{nd}(x_0,S_0)} }
\\
&\to
{\sum_{x \in S_0} K^{d - j}(u, x)
\hat{h}_0(x)}
\end{alignat*}
where we again used the finite support of $K^{d - j}(u, \cdot)$ to justify interchanging the limit and summation and then used \eqref{eqn:SurvivalRatioPeriodic}.

Combining the above, the l.h.s.\ of \eqref{eqn:p2} with $k = d$ has a limit
\begin{alignat}{2}
\pi^{-j}_u(y)
&\coloneqq
\frac
{\sum_{x \in S_0} K^{d - j}(u, x)
\pi^0_x(y) \hat{h}_0(x) }
{\sum_{z \in S_0} K^{d - j}(u, z)  {\hat{h}_0(z) }} \qt{for $y \in S_0$}
\\
&=
\sum_{x \in S_0} w(u,x) \pi^0_x(y) \qt{for $y \in S_0$}
\end{alignat}
where $w(u,x)$ is defined in \eqref{eqn:p35}.
Clearly, $w(u,x) > 0$ and $\sum_{x \in S_0} w(u,x) = 1$.
The limit $\pi^{-j}_u(y)$ is also nonnegative and $\sum_{y \in S_0} \pi^{-j}_u(y) = 1$.
Thus, after a temporary relabelling of $S_0, \dots, S_{d -1}$, we can use Prop.~\ref{periodicpro} to establish \eqref{eqn:p2} for all $k$.

Lastly, to establish \eqref{eqn:p4}, just use the result related to the
denominator of the l.h.s.\ of \eqref{eqn:p2} starting from $u$ and the analogous result starting from $v$.

\end{proof}

Recall that 
in the 
proof of Prop.~\ref{periodicpro}, 
we showed that 
$\prs{x}{\zeta > nd + k + 1 \given \zeta > nd +
k}\to \rho_k(x)$ 
where $x \in S_0$, which gives the asymptotic probability of surviving one
more step given that the process is currently in $S_k$ and started in state $x$.
Part of the hypothesis of the forthcoming Prop.~\ref{periodicproIII} will be that 
each of the functions $\rho_0(\cdot), \dots, \rho_{d - 1}(\cdot)$ is a constant;
that is,
\begin{alignat}{2}
	\rho_k(x) = \rho_k(y) \qt{for all $k$, and all $x,y \in S_0$.}
	\label{eqn:constantrho}
\end{alignat}
We will see an example where \eqref{eqn:constantrho} holds even though $\pi_x
\neq \pi_y$ whenever $x \neq y$.
Remark~\ref{rem:rho} gave one sufficient condition for---something stronger
than---\eqref{eqn:constantrho}
to hold.
A second sufficient
condition for \eqref{eqn:constantrho} is given in the following proposition.

\begin{proposition}
Let $K$ be irreducible, substochastic, and periodic with period $d > 1$.
If \eqref{eqn:PeriodicYaglomLimit1} holds for all $x \in S_0$ and if the
following limits exist 
\begin{equation}\label{eqn:SurvivalRatioPeriodic-again}
	\lim_{n \to \infty} \frac{K^n(x,S)}{K^n(x_0,S)} \quad\text{ for all $x \in
	S_0$},
\end{equation}
then \eqref{eqn:constantrho} holds for all $k \in \set{0, \ldots, d - 1}$.
\end{proposition}
\begin{proof}
	From the  
	proof of Prop.~\ref{periodicpro}, 
\begin{alignat*}{2}
\frac{K^{nd+k+1}(x,S_{k+1})}{K^{nd + k}(x,S_{k})} 
&\to
\rho_k(x) \qt{ for $x \in S_0$}
\end{alignat*}
where we used \eqref{eqn:PeriodicYaglomLimit1} and the corollary to Scheff\'{e}'s Theorem
in the last step.  
Consequently, for $x,y \in S_0$,
\begin{alignat*}{2}
	 \left(\frac{K^{nd+k+1}(x,S_{k+1})}{K^{nd+k}(x,S_{k})}\right) 
\left(\frac{K^{nd+k}(y,S_{k})} {K^{nd+k+1}(y,S_{k+1})}\right)
	&\to
	\frac{\rho_k(x)}{\rho_k(y)}.
	\intertext{On the other hand, the above limit could also be evaluated using
	\eqref{eqn:SurvivalRatioPeriodic-again}.  
	If the limits in \eqref{eqn:SurvivalRatioPeriodic-again} hold, then the
	limits must be equal to $\hat{h}_0(x)$ where $\hat{h}_0$ was defined in
\eqref{eqn:SurvivalRatioPeriodic}  Consequently,}  
	 \left(
		 \frac{K^{nd+k+1}(x,S_{k+1})}
{K^{nd+k+1}(y,S_{k+1})}
\right)
\left(\frac{K^{nd+k}(y,S_{k})} 
	 {K^{nd+k}(x,S_{k})}\right) 
&\to
\frac{\hat{h}_0(x)}{\hat{h}_0(y)}\cdot\frac{\hat{h}_0(y)}{\hat{h}_0(x)} =1, 
\end{alignat*}
which means that $\rho_k(x) = \rho_k(y)$ for all $k$.  
\end{proof}

If \eqref{eqn:SurvivalRatioPeriodic} holds, define $\hat{h}(u)$ for $u \in
S$ as follows:  if
$u \in S_j$, then 
\begin{equation}\label{eqn:periodic-h-hat}
\hat{h}(u) \coloneqq \rho^{j-d}\sum_{x\in
S_0}K^{d-j}(u,x)\hat{h}_0(x) 
\end{equation}
where $j \in \set{0, \ldots, d - 1}$.  
Similar to the aperiodic case, $\hat{h}$ is some sort of measure of the likelihood of surviving
for a long time relative to a fixed state $x_0$.  
The definition of $\hat{h}$ seems to be the correct way to extend $\hat{h}_0$
from a function on $S_0$ to a function on $S$ so that \eqref{eqn:pp3} holds and
so that $\hat{h}$ has a chance of being $\rho$-harmonic.  
The following proposition
gives conditions for $\hat{h}$ to be $\rho$-harmonic.  Although
there may be many other $\rho$-harmonic functions, $\hat{h}$ plays 
an important role in the large deviation behavior.  

\begin{proposition}\label{periodicproIII}
Let $K$ be irreducible, substochastic, and periodic with period $d > 1$.
If \eqref{eqn:PeriodicYaglomLimit1} holds for all $x \in S_0$,
\eqref{eqn:SurvivalRatioPeriodic} and
\eqref{eqn:constantrho} hold, and the support of $K(u,
\cdot)$ is finite for all $u \in S$, then
$\hat{h}$ is $\rho$-harmonic.  If in addition $u$ and $v$ are in the same class $S_j$, then
\begin{alignat}{2}
\frac{\hat{h}(v)}{\hat{h}(u)} &= \lim_{n \to \infty}
\frac{K^{n}(v,S)}{K^{n}(u,S)}.  
\label{eqn:pp3}
\end{alignat}
\end{proposition}
\begin{remark}The example in \autoref{subsec:remlife} shows that $\hat{h}$ may
	not be $\rho$-harmonic without the finite support assumption.
\end{remark}

\begin{proof}
	First, we show that $\hat{h}$ is $\rho$-harmonic.
	Fix $j \in \set{0, \ldots, d - 1}$.
	Let $u\in S_j$ and $v \in S_{j + 1}$ (where by convention $S_d$ means $S_0$).  
\begin{alignat*}{2}
\sum_v K(u,v)\hat{h}(v)
	&=\sum_vK(u,v)\rho^{j+1-d}\sum_{x\in S_0}K^{d-(j+1)}(v,x)\hat{h}_0(x)\\
	&=\rho\rho^{j-d}\sum_{x\in S_0}K^{d-j}(u,x)\hat{h}_0(x) \\
	&=\rho\hat{h}(u).
\end{alignat*}

Next, we prove the ratio limit result.  
Let $u \in S_j$ for fixed $j \in \set{0, \ldots, d - 1}$
From \eqref{eqn:constantrho}, $\rho_k(x)$ does not depend on
$x$, so we shorten it to $\rho_k$.  Fix $i \in \set{1, \ldots, d }$.
\begin{alignat*}{2}
	{K^{d - j + md + i}(u,S)} 
		&=
		\sum_{x,y\in S_0} K^{d-j}(u,x) K^{md}(x,y) K^i(y,S_i) \\
		\frac{K^{d - j + md + i}(u,S)}{K^{md}(x_0,S_0)}
		&= 
		\sum_{x,y\in S_0}
		K^{d-j}(u,x)\frac{K^{md}(x,y)}{K^{md}(x,S_0)}K^i(y,S_i)\frac{K^{md}(x,S_0)}{K^{md}(x_0,S_0)} \\
		&\to
		\sum_{x,y \in S_0} 
		K^{d-j}(u,x) \pi^0_x(y) \rho_0 \ldots \rho_{i - 1} \hat{h}_0(x) 
		\quad\text{as $m \to \infty$}	
		\\
		&= 
		\rho_0 \ldots \rho_{i - 1}
		\sum_{x \in S_0} K^{d - j}(u,x) \hat{h}_0(x) \\
		&= 
		\rho_0 \ldots \rho_{i - 1}
		\rho^{d - j} \hat{h}(u)
		\intertext{
where we justify interchanging the limit and sum in the next paragraph.  In the
above, we used the identity $\sum_{y \in S_0} \pi_x^0(y) K^j(y,S_j) = \rho_0(x)
\dots \rho_{j - 1}(x)$.  Also, since the $\rho_k$'s do not depend on $x$, they
could be factored outside the summation.
		Consequently, if $v$ is also in $S_i$, then}
	\frac{K^{d - j + md + i}(u,S)} 
	{K^{d - j + md + i}(v,S)} 
	&\to 
	\frac{\hat{h}(u)}{\hat{h}(v)} \quad\text{as $m \to \infty$.}
\end{alignat*}
Since this limit does not depend on $i$ and $j$ as long as $u$ and $v$ are in
the same class, \eqref{eqn:pp3} holds.  

Since 
there are only a finite number
of states $x$ accessible from $u$, it suffices to show that we can interchange
the limit and the sum over $y \in S_0$.  
Since 
$K^{md}(x,y)/K^{md}(x_0,S_0)$
converges to 
$\pi^0_x(y) \rho_0 \ldots \rho_{i - 1} \hat{h}_0(x),$ which is summable over $y
\in S_0$, and since $K^{i}(y, S_i)$ is bounded, we can again use the corollary
to Scheffe's theorem to justify 
interchanging the limit and sum.  
\end{proof}

\section{Duality and reversibility}\label{sec:dualwrev}
In some situations, a duality exists between $t$-invariant measures and 
$t$-harmonic functions.  
The example in \autoref{subsec:remlife} is a situation where they cannot be
linked since there is a $\rho$-invariant measure, but no $\rho$-harmonic
function.  We now describe a situation where such a duality arises and is
related to 
a kind of reversibility for
substochastic matrices; there will be 
additional duality discussion without reversibility in \autoref{sec:dualworev}.

First, we need several definitions.  
Given a positive $t$-harmonic function $h$ on $S$, Doob's $h$-transform of $K$,
sometimes 
called the \emph{twisted} kernel, is given by 
\begin{alignat}{2}
	\tilde{K}(x,y) &= \frac{K(x,y)h(y)}{t h(x)}. 
	\label{eqn:tildeKnew}
\intertext{
Similarly, given a positive $t$-invariant measure $\sigma$ on $S$, the time reversal
with respect to $\sigma$ is 
}
\ola{K}(x,y) &= 
\frac{\sigma(y) K(y,x)}{t\sigma(x)}.
\label{eqn:olaKnew}
\end{alignat}
Both $\ola{K}$ and $\tilde{K}$ are stochastic matrices on $S$.
There may be many different $t$-harmonic functions, and many different
$t$-invariant measures.
If  $\tilde{K} = \ola{K}$ and the same eigenvalue $t$ was used in constructing
both, then the $h$ used in constructing $\tilde{K}$ and the 
$\sigma$ used in constructing $\ola{K}$ are duals. 

We call $K$ \emph{reversible} if the Kolmogorov criterion holds; that is,
if
\begin{alignat}{2}
	K(x_0, x_1)K(x_1,x_2) \cdots K(x_{n - 1},x_n) &= K(x_n,
x_{n-1}) K(x_{n - 1}, x_{n - 2}) \cdots K(x_1,x_0) \label{st2}
\end{alignat} 
for any sequence of states 
$x_0, x_1, \ldots, x_n$ 
with $x_0 = x_n$. 
The following results are part of Theorems~4.1 in Pollett \cite{Pollett-1988}
and in Pollett\cite{Pollett-1989}. 
\begin{proposition}[Pollett\cite{Pollett-1989}]\label{prop:duality}
	Let $K$ be irreducible and substochastic on $S$.
$K$ is reversible iff 	
	there exists a positive measure $\gamma$ on $S$
			such that
			\begin{alignat}{2}
				\gamma(x)K(x,y) &= \gamma(y) K(y,x) \qt{ for all $x,y \in
				S$.}  \label{newst3}
			\end{alignat}
			If $\tilde{K} = \ola{K}$ where both were computed using the same
	eigenvalue $t$,	then $K$ is reversible. 
		If $K$ is reversible, $\sigma$ is $t$-invariant, $h$ is $t$-harmonic, and
		$\sigma(x) = h(x)/\gamma(x)$, then $\tilde{K} = \ola{K}$.
	If $K$ is reversible and $\sigma$ is a $t$-invariant measure, 
		 then $h(x) \coloneqq \sigma(x)/\gamma(x)$ for $x \in S$ defines 
			the dual 
			$t$-harmonic function.  
		Similarly, if $K$ is reversible and $h$ is a $t$-harmonic function, then $\sigma(x)
			\coloneqq h(x)\gamma(x)$ for $x \in S$ defines 
			the dual 
			$t$-invariant measure.  
\end{proposition}
\begin{remark}
	In the stochastic case, a constant function $h$ is $1$-harmonic so
	$\tilde{K} = \ola{K}$
	becomes the familiar $K$ = $\ola{K}$, though without the assumption of
	stationarity. 
	For example, consider a simple random walk on the integers that moves
	right with probability $0 < p < 1/2$. 
	Birth-death processes are reversible since they satisfy Kolmogorov's
	criterion, and for this process with $h$ being a column vector of ones, we
	have $\gamma(x) = \sigma(x) = (p/(1 - p))^x$.  For this pair,
	$\tilde{K}=\ola{K}$, and the associated Markov chains
	drift to negative infinity.
	There is another $1$-harmonic function: $((1 - p)/p)^x$. The corresponding
	$1$-invariant measure must be obtained by multiplying by $\gamma(x)$
	giving $(\ldots, 1,1,1, \ldots)$.  Both the time reversal and the
	twisted process for this pair drift to plus infinity.
	Notice that a
	single reversibility 
	measure $\gamma$ works for all eigenvalues $t$.
\end{remark}
%
\begin{proof}
	The following parts of the argument related to Kolmogorov's criterion is quite
	similar to the argument in the proof of Theorem~1.7 in \cite{Kelly}
	though without stationarity and for a substochastic matrix.  
	For any path $x = (x_0, x_1, \ldots, x_n)$, 
	let $\kappa(x) \coloneqq K(x_0, x_1)K(x_1, x_2) \cdots K(x_{n - 1},
	x_n)$.  
	In addition, let
	$\ola{x} = (x_n, x_{n - 1}, \ldots, x_0)$.  
	If \eqref{st2} holds and  
	$x_0 = x_n$, then $\kappa(x)
	= \kappa(\ola{x})$.  

	Let $y$ be some other path ending at $y_m$.  
	If
	\eqref{st2} holds, $x_0 = y_0$, $x_n = y_m$, and both paths have a
	positive probability of occurring, then 
	\begin{alignat}{2}
		\kappa(x)/\kappa(\ola{x}) &= \kappa(y)/\kappa(\ola{y}).
			\label{eqn:kapparatio}
	\end{alignat}	
	To see this, let $z$ be a path from $z_0 = x_n = y_m$ and ending at $z_\ell =
	x_0=y_0$ with $\kappa(z) > 0$.  Then $\kappa(xz) = \kappa(x)\kappa(z)$
	where $xz$ denotes the path that initially follows $x$ to $x_n$ and then follows
	$z$ back to $x_0$.  
	Similarly, $\kappa(yz) = \kappa(y)\kappa(z)$.  
	Under \eqref{st2}, $\kappa(xz) =
	\kappa(\ola{xz})=\kappa(\ola{z})\kappa(\ola{x})$; hence,
	$\kappa(x)/\kappa(\ola{x}) = \kappa(\ola{z})/\kappa(z)$.  
		Similarly, 
	$\kappa(y)/\kappa(\ola{y}) = \kappa(\ola{z})/\kappa(z)$, which gives
	\eqref{eqn:kapparatio}.

	Fix some state $0$ and $\gamma(0) > 0$.  Let $x$ be a path from $x_0=0$
	to some state $x_n$ with $\kappa(x) > 0$.  If \eqref{st2} holds, then 
	define $\gamma(x_n) \coloneqq \gamma(0)
	\kappa(x)/\kappa(\ola{x})$.  Under \eqref{st2}, it follows from
	\eqref{eqn:kapparatio} 
	that the definition of
	$\gamma(x_n)$ makes sense since the r.h.s.\  
	is the same for all such paths. 
	Consider a path that is one step longer:
	$x x_{n+1}$.  For this path, 
	\begin{alignat*}{2}
		\gamma(x_{n + 1}) &= \gamma(0) \frac{\kappa(x x_{n + 1})}
		{\kappa(\ola{x x_{n + 1}})} \\
		&=
		\gamma(0) \frac{\kappa(x)K(x_n,x_{n + 1})}
		{K(x_{n + 1},x_n)\kappa(\ola{x})} \\
		&= 
		\gamma(x_n) \frac{K(x_n,x_{n + 1})}
		{K(x_{n + 1},x_n)},  
	\end{alignat*}
	which means that \eqref{newst3} holds.  Thus, \eqref{st2} implies
	\eqref{newst3}.

	To see that \eqref{newst3} implies \eqref{st2}, suppose that $x$ is a path 
	$x = (x_0, x_1, \ldots, x_n)$.  Since $K(x_j, x_{j + 1}) = \gamma(x_{j
	+ 1}) K(x_{j + 1}, x_j)/\gamma(x_j)$, it follows that $\kappa(x) =
	\gamma(x_n)\kappa(\ola{x})/\gamma(x_0)$.  Letting $x_0 = x_n$ shows
	that Kolmogorov's criterion holds.  

	To see that $\tilde{K} = \ola{K}$ implies reversibility, where both
	were computed using the same eigenvalue $t$,   
	\begin{alignat*}{2}
		\tilde{K}(x,y) &= \ola{K}(y,x) \\
		\frac{K(x,y)h(y)}{th(x)}
		&=	
		\frac{\sigma(y)K(y,x)}{t\sigma(x)} \\
		\frac{\sigma(x) K(x,y)}{h(x)}
		&=	
		\frac{\sigma(y)K(y,x)}{h(y)}, 
	\end{alignat*}
	which means that \eqref{newst3} holds with $\gamma(x) = \sigma(x)/h(x)$.  
	
	The next claim that $\tilde{K} = \ola{K}$ under those conditions 
	follows from a straightforward algebraic simplification.  
	
	If $K$ is reversible, then we know that \eqref{newst3} holds with  
	$\gamma(x)$.  Since $\sigma$ is $t$-invariant, 
	\begin{alignat*}{2}	
	\ola{K}(x,y) 
	&=
	\frac
	{\sigma(y) K(y,x)}
	{t \sigma(x)} \\
	&= 
	\frac
	{ K(x,y) \sigma(y) / \gamma(y)}
	{t(\sigma(x)/\gamma(x))}.
	\end{alignat*}
	Since the sum over the l.h.s.\ is 1, $h=\sigma/\gamma$ is $t$-invariant, and
	$\tilde{K} = \ola{K}$.

	If $K$ is reversible, then we know that \eqref{newst3} holds with  
	$\gamma(x)$.  Since $h$ is $t$-harmonic, 
	\begin{alignat*}{2}	
	\tilde{K}(x,y) 
	&=
	\frac
	{K(x,y)h(y)}
	{th(x)} \\
	&= 
	\frac
	{\gamma(y)h(y)K(y,x)}
	{t\gamma(x)h(x)}.
	\end{alignat*}
	Since the sum over the l.h.s.\ is 1, $\sigma = \gamma h$ is $t$-invariant, and
	$\tilde{K} = \ola{K}$.
\end{proof}

\section{One idea that yields a handful of identities}\label{sec:identities}
We will exploit the following simple idea in computing various
quantities of interest:  frequently, it can be easier to  analyze a
well-chosen twist or time reversal rather than directly analyzing the process
of interest. For example, suppose we are interested in
computing the probability of ever hitting state $y$ starting from state $x$.  
If the well-chosen twist or time reversal eventually hits $y$ for certain when
starting from $x$---that is, if $\tilde{F}_{\xi^*}(x,y) = 1$ or
$\ola{F}_{\xi^*}(x,y) = 1$ in 
\eqref{eqn:hittingtwrev} below---then we have the hitting probabilities for 
other twisted processes and time reversals.  

Initially, assume that the process of interest is a Markov chain starting in 
state $x$ that
has transition matrix either 
$\ola{K}_\xi$ or $\tilde{K}_\xi$, where 
$\ola{K}_\xi$ would be the time reversal of $K$ with respect a $t$-invariant measure
$\sigma_\xi$ and $\tilde{K}_\xi$ would be the time reversal with respect to
$t$-harmonic function $h$.  Assume that the quantities of interest are the 
probably of ever hitting $y$ starting from $x$, which will be denoted by
$\ola{F}_\xi(x,y)$ and 
$\tilde{F}_\xi(x,y)$, respectively.  
Let either $\sigma_\xi^*$ be a (well-chosen) $t^*$-invariant measure, or let
$h_\xi^*$ be a (well-chosen) $t^*$-harmonic function. For simplicity, assume
$t^* = t$.  
\begin{alignat}{2}
	\ola{K}_{\xi}(x,y) &=
	\frac{\sigma_\xi(y)}{\sigma_{\xi^*(y)}}
	\frac{\sigma_{\xi^*(x)}}{\sigma_\xi(x)}
	\ola{K}_{\xi^*}(x,y)
	&\text{ and } 
	\tilde{K}_{\xi}(x,y) &=
	\tilde{K}_{\xi^*}(x,y)
	\frac{h_{\xi^*(x)}}{h_\xi(x)}
	\frac{h_\xi(y)}{h_{\xi^*(y)}} \\
	\ola{F}_{\xi}(x,y) &=
	\frac{\sigma_\xi(y)}{\sigma_{\xi^*(y)}}
	\frac{\sigma_{\xi^*(x)}}{\sigma_\xi(x)}
	\ola{F}_{\xi^*}(x,y)
	&\text{ and } 
	\tilde{F}_{\xi}(x,y) &=
	\tilde{F}_{\xi^*}(x,y)
	\frac{h_{\xi^*(x)}}{h_\xi(x)}
	\frac{h_\xi(y)}{h_{\xi^*(y)}}. \label{eqn:hittingtwrev}
\end{alignat}
The idea is simply to undo the twist or reverse and redo the twist or reverse
with a better measure or harmonic function.  In some contexts, it might be advantageous
to mix the two: undoing the twist and then applying a time reversal or
vice-versa.  

The same idea can be exploited to obtain useful expressions for a different
quantity of interest:  
the generating function $G_{(x,y)}(z)$ defined in \eqref{eqn:GenFun}.  
Suppose that $t = 1/z$ and that $\sigma_\xi^*$ is a (well-chosen) $t$-invariant measure or
$h_\xi^*$ is a (well-chosen) $t$-harmonic function for $K$.  Then 
\begin{alignat}{2}
	G_{(x,y)}(z) 
	&= \frac{\sigma_{\xi^*}(y)}{\sigma_{\xi^*}(x)} \ola{G}_{\xi^*}(y,x)
	\label{eqn:revG} \\
	&= \tilde{G}_{\xi^*}(x,y) \frac{h_{\xi^*}(y)}{h_{\xi^*}(x)} \label{eqn:twG}
\end{alignat}
where $\ola{G}_{\xi^*}(y,x)$ is the expected number of visits to $x$ starting from $y$
for the reversed process with transition matrix $\ola{K}_{\xi^*}$ and
$\tilde{G}_{\xi^*}(x,y)$
is the expected number of visits to $y$ starting from $x$ for the twisted
process with transition matrix $\tilde{K}_{\xi^*}$.
The proofs of the above equations are straightforward algebraic
manipulations.

\section{Examples}\label{sec:examples}
The examples are based on the Seneta and Vere-Jones'~\cite{Vere-Jones-Seneta}
semi-infinite random walk with absorption, which is the same as the
gambler's ruin problem in the Introduction.
Our primary example is a ``hub-and-two-spoke example'' that is depicted in Fig.~\ref{fig:OurExample}.
A hub-and-one-spoke model, shown in Fig.~\ref{fig:graphSVrelabelled}, functions
as a \emph{notational} bridge between our primary example and the
Seneta--Vere-Jones example.
The hub-and-one-spoke example is simply a relabeled version of the
Seneta--Vere-Jones example.
The notation for these 3
examples will be the following:
\begin{description}
\item  [Hub-and-two-spoke:]
	Let $X= \set{X_0, X_1, \dots}$ be a Markov chain with state
	space $\ZZ \coloneqq \set{\dots, -1, 0, 1, \dots}$ augmented
	with an absorbing state $\delta$ and transition matrix $K$
	between states in $\ZZ$ as shown in Figure~\ref{fig:OurExample}.
\item [Hub-and-one-spoke:] Let $Y =
	\set{Y_0, Y_1, \dots}$ denote the Markov chain with state space
	$\NN_0$ augmented by an absorbing state
	$\delta$, and $Q$ will denote the transition matrix between states in $\NN_0$ as shown in Figure~\ref{fig:graphSVrelabelled}.
    \item[Seneta--Vere-Jones:] Let $Z = \set{Z_0, Z_1, \dots}$ be the
	    Markov chain with state space $\NN$
	    augmented by an additional absorbing state 0.
	    The strictly substochastic matrix $P$ gives the transition
	    probabilities between states in $\NN$ where
    \begin{align}
P &=
\begin{bmatrix}
0 & b & 0 & 0 & 0 & \cdots \\
a & 0 & b & 0  & 0 & \cdots \\
0 & a & 0 & b & 0 & \cdots \\
\vdots
\end{bmatrix} \label{eqn:P}
\end{align}
\end{description}
Throughout, we assume that $0 < b < 1/2 < a < 1$ and $a + b = 1$.

\begin{figure}
\begin{tikzpicture}[->,>=stealth',shorten >=1pt,auto,node distance=2cm,
  thick,main node/.style={circle,fill=blue!20,draw,font=\sffamily\Large\bfseries}]

  \node[main node] (0) {0};
  \node[main node] (d) [below of= 0] {$\delta$};
  \node[main node] (1) [right of=0] {1};
  \node[main node] (-1) [left of=0] {-1};
  \node[main node] (2) [right of=1] {2};
  \node[main node] (-2) [left of=-1] {-2};
  \node[right of=2](dp){$\dots$};
  \node[left of=-2](dn){$\dots$};

  \path[every node/.style={font=\sffamily\small}]
    (0) edge [bend left] node[above] {$b/2$} (1)
        edge [bend right] node[above] {$b/2$} (-1)
        edge node [left] {$a$} (d)

    (1) edge [bend left] node[above] {$b$} (2)

        edge [bend left] node[below] {$a$} (0)
    (-1) edge [bend right] node[above] {$b$} (-2)
        edge [bend right] node[below] {$a$} (0)
    (2) edge [bend left] node[above] {$b$} (dp)
        edge [bend left] node[below] {$a$} (1)
    (-2) edge [bend right] node[above] {$b$} (dn)
        edge [bend right] node[below] {$a$} (-1)
    (dp) edge [bend left] node[below] {$a$} (2)
    (dn) edge [bend right] node[below] {$a$} (-2)
    (d) edge [loop below] (d)

;
\end{tikzpicture}
\caption{$K$  is restricted to $\ZZ$.}\label{fig:OurExample}
\end{figure}

\begin{figure}
\begin{tikzpicture}[->,>=stealth',shorten >=1pt,auto,node distance=2cm,
  thick,main node/.style={circle,fill=blue!20,draw,font=\sffamily\Large\bfseries}]

  \node[main node] (0) {0};
  \node[main node] (d) [below of= 0] {$\delta$};
  \node[main node] (1) [right of=0] {1};
  \node (-1) [left of=0] {};
  \node[main node] (2) [right of=1] {2};
  \node (-2) [left of=-1] {};
  \node[right of=2](dp){$\dots$};
  \node[left of=-2](dn){};

  \path[every node/.style={font=\sffamily\small}]
    (0) edge [bend left] node[above] {$b$} (1)
        edge node [left] {$a$} (d)
    (1) edge [bend left] node[above] {$b$} (2)
        edge [bend left] node[below] {$a$} (0)
    (2) edge [bend left] node[above] {$b$} (dp)
        edge [bend left] node[below] {$a$} (1)
    (dp) edge [bend left] node[below] {$a$} (2)
     (d) edge [loop below] (d)
;
\end{tikzpicture}
\caption{$Q$  is restricted to $\NN_0$.}\label{fig:graphSVrelabelled}
\end{figure}

These three examples can be coupled in the following natural way.  Given the
hub-and-two-spoke model $X$, let
$Y_n = \abs*{X_n}
$
and
$Z_n = (Y_n + 1)
$ for all $n$ prior to the (common) time of absorption $\zeta$.
The coupling makes it easier to take advantage of results
in Seneta--Vere-Jones~\cite{Vere-Jones-Seneta}.  For example, $K^n(0,0) = Q^n(0,0) = P^n(1,1)$ so all
three matrices have the same convergence parameter $R$, and from
Seneta--Vere-Jones~\cite{Vere-Jones-Seneta}
$R = 1/\rho$ where $\rho = 2 \sqrt{ab}$.
Next, we review  a few results from
Seneta--Vere-Jones~\cite{Vere-Jones-Seneta}.

\subsection{Seneta--Vere-Jones semi-infinite random walk with
absorption}\label{subsec:SVJRW}
The matrix
$P$ given in \eqref{eqn:P} for the Seneta--Vere-Jones example~\cite{Vere-Jones-Seneta} is irreducible, strictly substochastic, and periodic with period 2.
Let $f_n$ be the probability that the first return to state 1 starting from 1 occurs at time $n$.  Then the generating function $F(z) = \sum_{n \geq 0} f_n z^n =
\left(1-\sqrt{1-4abz^2}\right)/2$.   Hence, the convergence parameter of $P$ is
$R = 1/\rho$ where $\rho = 2\sqrt{ab}$. Since $F(R)=1/2$, $G_{1,1}(R) = 1/(1 -
F(R)) = 2 < \infty$, $P$ must be $R$-transient.
%

Seneta and Vere-Jones~\cite{Vere-Jones-Seneta} prove a periodic Yaglom limit where the \rhoi \qsd\ on the r.h.s.\ is $\pi^*$ given in \eqref{eqn:SVLimit}, which does not depend on the starting state $x$.  From (35)~in~\cite{Vere-Jones-Seneta} and from p.~430 of~\cite{Vere-Jones-Seneta}, we have the following asymptotic expressions as $n \to \infty$
\begin{alignat}{2}
P^{2n}(x,y)
 &\sim x \left(\sqrt{\frac{a}{b}}\right)^{x-1} y \left(\sqrt{\frac{b}{a}}\right)^{y-1}
 \sqrt{\frac{1}{\pi }} \frac{(4ab)^n}{ n^{3/2}}
 \qt{for $y - x$ even} \label{eqn:SVAP}
 \\
\pr{\zeta = n \given Z_0 = x}
&= \frac{x}{n} \binom{n}{(n - x)/2} b^{(n - x)/2} a^{(n + x)/2} \qt{for $n - x$ even}
\label{eqn:asymptau}
\\
&\sim \frac{x\cdot 2^{n+1}}{(2\pi)^{1/2}(n)^{3/2}} b^{\frac{1}{2}(n-x)}a^{\frac{1}{2}(n+x)}
\qt{for $n - x$ even}  \label{eqn:SVAZ}
\end{alignat}
where $\zeta$ is the time of absorption.

\subsection{Hub-and-two-spoke model: our primary example}\label{subsec:ModEx}

At first, this two-spoke variation may seem pointless, but the point is to
construct a tractable model that has more
than one way to escape from 0.  Each spoke provides a different escape route.
%
We will show that the periodic Yaglom limit starting from state $x$ in the hub-and-two-spoke example is
\begin{equation}\label{eqn:periodicYaglomExample}
\pi_x(y) =
\begin{cases}
\frac{1 - \rho}{2a} \left(1 + \abs*{y} + \frac{x}{1 + \abs*{x}  } y \right)
\left( \sqrt{\frac{b}{a}} \right)^{\abs*{y}} &\text{for $y \in \ZZ\setminus\set{0}$}
\\
\frac{1 - \rho}{a}  &\text{for $y = 0$.}
\end{cases}
\end{equation}
For $y > 0$, $\pi_x(y)$ is strictly increasing in $x$---each starting state $x$
has a different Yaglom limit.  For aperiodic examples, it suffices to look at
either the even states or the odd states and use the two-step transition matrix
$K^2$.  For $K^2$, the limiting conditional
distribution of being in state $2y$ for $y > 0$ is strictly increasing in the starting
state $2x$.

\begin{theorem}\label{mainone}
The hub-and-two-spoke model with $0 < b < 1/2 < a < 1$ and $a + b = 1$ is
periodic with period $d = 2$ and has a
periodic Yaglom limit $\pi_x$ given in \eqref{eqn:periodicYaglomExample}.
Equivalently, 
\begin{alignat}{2}
\frac{K^{2n}(x,y)}{K^{2n}(x,S)}
&\to
\frac{\pi_x(y)}{\pi_x(2 \ZZ)} \qt{for $x$ even and $y$ even,} \label{eqn:ee}
\\
\frac{K^{2n + 1}(x,y)}{K^{2n + 1}(x,S)}
&\to \frac{\pi_x(y)}{\pi_x(2 \ZZ + 1)} \qt{for $x$ even and $y$ odd,} \label{eqn:eo}
\\
\frac{K^{2n}(x,y)}{K^{2n}(x,S)}
&\to \frac{\pi_x(y)}{\pi_x(2 \ZZ + 1)} \qt{for $x$ odd and $y$ odd,} \label{eqn:oo}
\\
\frac{K^{2n + 1}(x,y)}{K^{2n + 1}(x,S)} &\to \frac{\pi_x(y)}{\pi_x(2 \ZZ)} \qt{for $x$ odd and $y$ even} \label{eqn:oe}
\intertext{where}
\pi_x(2\ZZ) &= \frac{1}{1 + \rho} \label{eqn:even} \\
\pi_x(2 \ZZ + 1) &= \frac{\rho}{1 + \rho}. \label{eqn:odd}
\end{alignat}
\end{theorem}

In the next two sections, we prove Theorem~\ref{mainone}.
To show that \eqref{eqn:ee}--\eqref{eqn:oe} hold, we first look at the
asymptotics of their denominators,  and then the asymptotics for their
numerators.
In the remainder of this section, we describe the $\rho$-invariant measures and harmonic functions for the
hub-and-two-spoke model.

We leave it to the reader to show that the matrix $K$ possesses a family of 
$\rho$-invariant \qsds\  $\sigma_\xi$ indexed by $\xi \in [-1, 1]$ and given by
\begin{equation}\label{eqn:measures}
\sigma_\xi(y) =
\begin{cases}
\frac{1 - \rho}{2a} \left(1  + \abs*{y} + \xi \, y \right) \left( \sqrt{\frac{b}{a}} \right)^{\abs*{y}} &\text{for $y \in \ZZ\setminus\set{0}$}
\\
\frac{1 - \rho}{a} &\text{for $y = 0$.} 
\end{cases}
\end{equation}
Thus, if the chain starts from state $x$,  $\pi_x = \sigma_\xi$ with $\xi
= x / (1 + \abs*{x})$ is the periodic Yaglom limit.
The derivation of the asymptotics of $K$ is given in the next sections.

By looking at $\sigma_\xi(1)$, it is clear that every member of the family is a
different distribution.  For each distribution, the correct amount of mass is
lost to absorption: $a \sigma_\xi(0) = 1 - \rho$.  As $\xi$ increases from $-1$
to 1, $\sigma_\xi(\abs*{y})$ increases from $((1 - \rho)/(2a))(\sqrt{b/a})^{\abs*{y}}$ to
$((1 - \rho)/(2a))(1 + 2\abs*{y})(\sqrt{b/a})^{\abs*{y}}$.  When $\xi = 0$, the
distribution is symmetric with $\sigma_0(y) = \sigma_0(-y)$.
The mass $\sigma_\xi(y) + \sigma_\xi(-y) = \pi^*(\abs*{y} + 1)$ does not depend on $\xi$;  $\pi^*$ was defined in \eqref{eqn:SVLimit}.   Consequently, the mass on the even integers $2\ZZ$ and the odd integers $2\ZZ + 1$ does not depend on $\xi$.  Since
\begin{alignat*}{2}
\sigma_\xi K (2 \ZZ) &= \rho \,\sigma_\xi (2 \ZZ) \qt{ by $\rho$-invariance, and} \\
\sigma_\xi K (2 \ZZ) &= \sigma_\xi (2 \ZZ + 1)  \qt{ by periodicity and nonabsorption},
\intertext{we have}
\sigma_\xi (2 \ZZ) &= \frac{1}{1 + \rho},
\sigma_\xi (2 \ZZ + 1) = \frac{\rho}{1 + \rho}
\end{alignat*}

In addition to the $\rho$-invariant measures $\sigma_\xi$, $K$ also
has $\rho$-harmonic functions $h_\xi$.
 A function $h \geq 0$, which we think of as a column vector with elements $h(y)$ for $y \in S$, is $\rho$-harmonic if
$Kh = \rho h$.

Equivalently, $h$ is a nonnegative right eigenvector for the eigenvalue $\rho =
1/R$.  In this example,
$K$ has a family of nonnegative $\rho$-harmonic functions $h_\xi$ indexed by $\xi \in [-1, 1]$:
\begin{alignat}{2}
h_\xi(y) &\equiv \left[1 + \abs*{y} + \xi y \right] \left( \sqrt{\frac{a}{b}} \right)^{\abs*{y}} \qt{for $y \in \ZZ$}. \label{eqn:harmonic}
\end{alignat}

\begin{proposition}
All $\rho$-invariant probability measures for the hub-and-two-spokes example
are in the family \eqref{eqn:measures} and all positive $\rho$-harmonic
functions are in the family \eqref{eqn:harmonic}.
\end{proposition}

\begin{proof}
To see that there are no other positive $\rho$-harmonic functions $h$ with
$h(0) = 1$, notice that $h(y)$ satisfies the difference equation $b h(y + 1) -
\rho h(y) + a h(y - 1)=0$ for $y \in \NN$.  Since this difference equation is a
linear, homogeneous, second order difference equation with constant
coefficients, we look at the roots of the characteristic equation $b r^2 - \rho
r + a$ to determine the general solution.  Both roots are $r = \sqrt{a/b}$, so
$h(y) = (c_1 + c_2 y) r^y$ spans the space of all solutions.  Since $h(0) = 1$,
we must have $c_1 = 1$, which we do in \eqref{eqn:harmonic}.  For $h(y)$ to be
nonnegative, we need $c_2 \geq 0$.  For $y \geq 0$, that means that $(1 +
\xi)=c_2 \geq 0$ implying $\xi \geq -1$.  By symmetry, we need $\xi \leq 1$ so
that $h(y)$ is nonnegative for $y < 0$.  Thus, \eqref{eqn:harmonic} is the set
of all positive harmonic functions normalized to have $h(0) = 1$.  A similar
difference equation argument shows that $\sigma_\xi$ for $\xi \in [-1,1]$
describes all $\rho$-invariant measures up to scalar multiples.
\end{proof}

In \autoref{sec:dualwrev} we described a duality that sometimes exists between
$\rho$-invariant measures and $\rho$-harmonic functions.  
We now show that for each value of $\xi \in [-1,1]$ the $\rho$-harmonic function
$h_\xi$ and the $\rho$-invariant measure $\sigma_\xi$ are linked.  
For the hub-and-two-spokes
example, 
the measure 
\begin{equation}\label{eqn:h2gamma}
	\gamma(x) = 
\begin{cases}
	\frac{1 - \rho}{2a}\left(\frac{b}{a}\right)^{|x|} &\text{ for $x
	\neq 0$} \\
	\frac{1 - \rho}{a} &\text{ for $x = 0$.}
\end{cases}
\end{equation}
satisfies $\gamma(x)K(x,y)=\gamma(y)K(y,x)$. 
From Prop.~\ref{prop:duality}, the $\rho$-invariant measure linked with the
$\rho$-harmonic function 
$h_\xi(x)$ would be 
$h_\xi(x)\gamma(x)$, which simplifies to $\sigma_\xi(x)$.


\subsubsection{Asymptotic behavior of the survival probability.}
For $x$ even, $K^{2n}(x,S) = K^{2n - 1}(x,S)$.
Since $K^n(-x,S) = K^n(x,S)$, also assume that $x$ is nonnegative.
Now,
\begin{alignat}{2}
K^{2n}(x,S)
&=\pr{\zeta > 2n \given X_0 = x}
\\
&= \sum_{k = 1}^\infty \pr{\zeta = 2n + 2k - 1 \given X_0 = x} \notag
\\
&= \sum_{k = 1}^\infty \pr{\zeta = 2n + 2k - 1 \given Z_0 = x + 1}.  \notag
\intertext{Since $\pr{\zeta = 2n + 2k - 1 \given Z_0 = x + 1}$ is asymptotically equivalent
to the r.h.s.\ of \eqref{eqn:SVAZ} after replacing $n$ by $2n + 2k - 1$ and $x$
by $x + 1$ and since the equivalence is uniform over $k \in \NN$ (for a
discussion of this condition,  see (5.27) of~\cite{Odlyzko95}),}
K^{2n}(x,S)
&\sim \sum_{k = 1}^\infty \frac{(x + 1)  2^{2n + 2k}  }{(2\pi)^{1/2}(2n + 2k - 1)^{3/2}} { b^{ (2n + 2k - x - 2)/2}a^{ (2n + 2k  + x)/2}} \notag
\\
&=
\frac{(x+1)}{(2\pi)^{1/2}}\left(\sqrt{\frac{a}{b}}\right)^{x}\frac{(4ab)^n}{(2n)^{3/2}}
\frac{1}{b} \sum_{k = 1}^\infty \frac{(4ab)^k}{(1 + (2k - 1)/(2n))^{3/2}}
\notag
\\
&\sim \frac{(x+1)}{(2\pi)^{1/2}}\left(\sqrt{\frac{a}{b}}\right)^{x}\frac{(4ab)^n}{(2n)^{3/2}}
\frac{4a}{1-4ab}
\notag
\\
&= {(x+1)}\left(\sqrt{\frac{a}{b}}\right)^{x}  \frac{a}{1-4ab}\frac{1}{\sqrt\pi}\frac{(4ab)^n}{n^{3/2}}
\quad\text{ for even $x\geq0$}
\label{eqn:alive}
\end{alignat}
where the second from last step follows from using dominated convergence to show that
\[
	\lim_{n \to \infty} \sum_{k = 1}^\infty 
	\frac{(4ab)^k}{(1 + (2k - 1)/(2n))^{3/2}} = \sum_{k = 1}^\infty (4ab)^k.
\]



%
%


\subsubsection{The periodic Yaglom limit of the
hub-and-two-spoke model.}
In this section, we show $K$ corresponding to Figure~\ref{fig:OurExample} has a periodic Yaglom limit; that is, we establish \eqref{eqn:periodicYaglomExample}.
Let $x,y \geq 0$. From the coupling, we have
\begin{alignat*}{2}
Q^n(x,y)
&= K^n(x,-y) + K^n(x, y),
\end{alignat*}
and we can determine the asymptotics of $Q$ from the asymptotics of $P$
given in \eqref{eqn:SVAP} that were derived in Seneta and Vere-Jones~\cite{Vere-Jones-Seneta}.

Similar to the classical ballot problem, there are two types of paths from $x$
to $y$: those that visit 0 and those that do not.  From the reflection
principle, any path from $x$ to $y$ that visits 0 has a corresponding path from
$-x$ to $y$ with the same probability of occurring.  Thus, if $\zK^n(x,y)$
denotes the probability of going from $x$ to $y$ in $n$ steps without
visiting zero in between, we have
\begin{alignat*}{2}
K^n(x,y) &= \zK^n(x,y) + K^n(-x,y)=\zK^n(x,y) + K^n(x,-y).
\end{alignat*}
From the coupling, $\zK^n(x,y) = P^n(x,y)$.

 For $x \geq 0$ and $y \geq 1$,
\begin{alignat*}{2}
Q^n(x, y) &=  K^n(x, y) + K^n(x, -y).
\end{alignat*}
Hence
\begin{alignat}{2}
 K^n(x, y) &=  Q^n(x, y) - K^n(x, -y)\\
           &= Q^n(x,{y}) - \left[ K^n(x,y)- \zK^n(x,y) \right] \\
           &= \frac{ Q^n(x,y) + \zK^n(x,y)}{2}. \label{eqn:Kp}
\intertext{Similarly,}
 K^n(x, -y)    &= \frac{  Q^n(x,{y})-\zK^n(x,y)}{2}. \label{eqn:Kn}
\end{alignat}

\subsubsection{Yaglom limits from even states to even.}\label{subsubsec:eventoeven}
If either $x$
is 0 or $y$ is 0, then the asymptotics of $K^n(x,y)/K^n(x,S)$ can be obtained
directly from the results in Seneta and Vere-Jones~\cite{Vere-Jones-Seneta} through the coupling
Thus, \eqref{eqn:PeriodicYaglomLimit2} holds for $k = 0$.

Now, let $x,y \geq 1$ and even.  From the couplings,
\begin{alignat*}{2}
Q^{2n}(x,y) &= P^{2n}(x + 1, y + 1),
\zK^{2n}(x,y)=P^{2n}(x,y).
\intertext{Since $x - y$ is also even, \eqref{eqn:SVAP} gives}
Q^{2n}(x,y) &\sim (x+1)\left(\sqrt{\frac{a}{b}}\right)^{x}(y+1)\left(\sqrt{\frac{b}{a}}\right)^{y}
{\frac{1}{\sqrt\pi}}\frac{(4ab)^n}{n^{3/2}}  \\
\zK^{2n}(x,y) &\sim x \left(\sqrt{\frac{a}{b}}\right)^{x-1} y \left(\sqrt{\frac{b}{a}}\right)^{y-1}
 {\frac{1}{\sqrt\pi }} \frac{(4ab)^n}{ n^{3/2}}
\end{alignat*}

Hence, using \eqref{eqn:Kp} and \eqref{eqn:alive},
\begin{alignat}{2}
\frac{K^{2n}(x,y)}{K^{2n}(x,S)} &\sim \frac{(x+1)\left(\sqrt{\frac{a}{b}}\right)^{x}(y+1)\left(\sqrt{\frac{b}{a}}\right)^{y}
{\frac{1}{\sqrt\pi}}\frac{(4ab)^n}{n^{3/2}}}{2{(x+1)}\left(\sqrt{\frac{a}{b}}\right)^{x}  \frac{a}{1-4ab}\frac{1}{\sqrt\pi}\frac{(4ab)^n}{n^{3/2}}}
\notag
\\
&+ \frac{x \left(\sqrt{\frac{a}{b}}\right)^{x-1} y \left(\sqrt{\frac{b}{a}}\right)^{y-1}
 {\frac{1}{\sqrt\pi }} \frac{(4ab)^n}{n^{3/2}}}{2{(x+1)}\left(\sqrt{\frac{a}{b}}\right)^{x}  \frac{a}{1-4ab}\frac{1}{\sqrt\pi}\frac{(4ab)^n}{n^{3/2}}}
\notag
\\
&= \frac{1-4ab}{2a} \left[ (y+1) + \frac{xy}{x + 1} \right] \left(\sqrt{\frac{b}{a}}\right)^{y}
\label{eqn:etoeasymp}
\\
&= \frac{1-\rho^2}{2a} \left( 1 + y  + \frac{xy}{x + 1} \right) \left(\sqrt{\frac{b}{a}}\right)^{y}.
\label{eqn:pi0xy}
\end{alignat}

We now argue that \eqref{eqn:PeriodicYaglomLimit2} holds for $d=2$, $k=0$ and $S_0 = 2\ZZ$ with
\begin{equation}\label{eqn:pi0xybetter}
\pi_x^0(y) =
\begin{cases}
\frac{1-\rho^2}{2a} \left( 1 + \abs*{y}  + \frac{xy}{\abs*{x} + 1} \right) \left(\sqrt{\frac{b}{a}}\right)^{\abs*{y}} &\text{for $y \neq 0$ and even $x,y$,}
\\
\frac{1-\rho^2}{a} &\text{for $y = 0$ and even $x$.}
\end{cases}
\end{equation}
We are still assuming even $x,y \geq 1$.
Since $\pi_x^0(y)$ is the same as \eqref{eqn:pi0xy}, we know that \eqref{eqn:pi0xybetter} holds for even $x,y \geq 1$.
The asymptotics for $-x$ to $-y$ are the same as from $x$ to $y$, and $\pi_{-x}^0(-y) = \pi_x^0(y)$ so \eqref{eqn:pi0xybetter} also holds in this case.

To handle from $x$ to $-y$, we use \eqref{eqn:Kn} instead of \eqref{eqn:Kp},
which causes a single sign change, and the final result agrees with
\eqref{eqn:pi0xybetter} in this case.  The asymptotics from $-x$ to $y$ are the
same as from $x$ to $-y$, and \eqref{eqn:pi0xybetter} gives the same result in
both cases.

\subsubsection{Yaglom limits from even states to odd.}

Instead of using a similar argument for the asymptotics from even to odd, we
use Proposition~\ref{periodicpro}.  Since \eqref{eqn:PeriodicYaglomLimit2}
holds for $k = 0$, Proposition~\ref{periodicpro} gives
\eqref{eqn:PeriodicYaglomLimit2} for $k = 1$.  Proposition~\ref{periodicpro}
also gives $\pi_x^1$, which is a probability measure on the odd states $S_1 = 2
\ZZ - 1$ giving the asymptotics from even to odd.
Next, Proposition~\ref{periodicpro} gives us
the \rhoi\qsd\ $\pi_x$ for every even $x$.
We leave it to the reader to show that $\pi_x$ is given by \eqref{eqn:periodicYaglomExample}.
Hence, we have the asymptotics
going to any state as long as the
starting state is even.

\subsubsection{Yaglom limits starting from odd states.}\label{sec:hath}
To finish determining the periodic Yaglom limit, we obtain the asymptotics starting from an odd state.
Instead of direct calculations like
those that led to \eqref{eqn:etoeasymp},
We use Prop.~\ref{periodicproII}.
To do so, we need the function $\hat{h}_0$ defined in
\eqref{eqn:SurvivalRatioPeriodic}.  The class $S_0$ is the even states.
Choose $x_0 = 0$ to be the reference state.
From \eqref{eqn:alive} and symmetry, it follows that
$\hat{h}_0(x)= (\abs{x} + 1) \left( \sqrt{ a/b } \right)^{\abs{x}}$.
Since the assumptions of Prop.~\ref{periodicproIII} hold, there exists a $\rho$-harmonic
function $\hat{h}$ that agrees with $\hat{h}_0$ on $S_0$ such that
\eqref{eqn:pp3} holds.
However, \eqref{eqn:harmonic} describes all $\rho$-harmonic functions for $K$.
The only $\rho$-harmonic function that could agree with $\hat{h}$ on $S_0$ is $h_0$, that is, $h_\xi$
with $\xi = 0$ in \eqref{eqn:harmonic}.
Thus, $\hat{h} = h_0$.  Furthermore, since the assumptions of Prop.~\ref{periodicproIII} hold, the denominator of
\eqref{eqn:p35} simplifies to $\rho^{d - j}\hat{h}(u)$.

Let $u \in S_1$, that is, an odd state, and $y \in S_0$.
To make things easier, temporarily assume $u \geq 1$ and $y \neq 0$.
From \eqref{eqn:p3},
\begin{alignat*}{2}
	\pi_u^{-1}(y) &= \sum_{x \in S_0} w_{u,x} \pi_x^0(y) \\
	&= \frac{a \hat{h}(u - 1)}{\rho \hat{h}(u)} \pi_{u - 1}^0(y)
	+ \frac{b \hat{h}(u + 1)}{\rho \hat{h}(u)} \pi_{u + 1}^0(y) \\
	&= \frac{a u}{2 \sqrt{a b} (u + 1)} \sqrt{\frac{b}{a}} \, \pi_{u - 1}^0(y)
	+ \frac{b (u + 2)}{2\sqrt{a b} (u + 1)} \sqrt{\frac{a}{b}} \, \pi_{u + 1}^0(y) \\
	&= \frac{u}{2(u + 1)} \pi_{u - 1}^0(y)
	+  \frac{u + 2}{2(u + 1)} \pi_{u + 1}^0(y) \\
	&= \frac{1 - \rho^2}{2 a} \left( 1 + \abs{y} + \left[\frac{u}{2(u + 1)} \frac{u - 1}{u} + \frac{u + 2}{2(u + 1)} \frac{u + 1}{u + 2} \right] y \right) \left( \sqrt{\frac{b}{a}}  \right)^{\abs{y}}  \\
	&= \pi_u^0(y)
\end{alignat*}
where we used \eqref{eqn:pi0xybetter} several times.  The case with $y = 0$ is much simpler and also simplifies to $\pi_u^0(0)$.
We leave the cases with $u \leq 1$ to the reader.
Thus, we have the asymptotics starting from an odd state and going to an even state.  
Again, Prop.~\ref{periodicpro} allows us to extend the result to
going to any state.  Hence, we have the asymptotics starting from any state and going to any state.  Combining all of the results, we have that
\eqref{eqn:PeriodicYaglomLimit0} holds where $\pi_u(y)$ is given in \eqref{eqn:periodicYaglomExample}
(though we would have to interchange labels on $S_0$ and $S_1$ in \eqref{eqn:PeriodicYaglomLimit0}  if the
initial state were odd).

\subsubsection{Rates of convergence and the starting state's influence.}

One fear with Yaglom limits is that by the time the transient conditional
distribution becomes close to the limiting conditional distribution, the
probability of non-absorption will be so small that the limit will be of little
practical importance.  To briefly address this fear, we describe some
empirical results where the dependence of the limiting distribution on the initial
state is apparent after a small number of steps (50 steps in Table~\ref{Table:from10})
and the non-absorption probability is not ridiculously small ($0.00047$ in Table~\ref{Table:from10}).

Suppose that $b = 1/5$, so $a = 4/5$ and $\rho = 4/5$.  Also assume
that  the initial state is $X_0 = 10$.  From \eqref{eqn:pi0xybetter}, $\prs{10}{X_{2n} = 0
\given \zeta > 2n} \to \pi_{10}^0(0)  = 0.45$ and $\prs{10}{X_{2n} \in 2\NN  \given \zeta >
2n} \to 201/440 \approx 0.46$.  The latter limit is the limiting conditional
probability starting from 10 of being a strictly positive, even integer after
an even number of steps.  Starting from state 10, $\xi = 10/11$;
asymptotically after a large, even number of steps, over $90\%$ of the probability
mass is on the nonnegative, even integers.  The remaining probability mass,
approximately 0.09, is on the strictly negative, even integers.
Table~\ref{Table:from10} suggests that at least in some cases the limiting
conditional distribution might be giving some information before the
probability of non-absorption $K^n(10,S)$ becomes ridiculously smal1.

Table~\ref{Table:from10} also illustrates the long range influence of the
starting state.  The limiting conditional probability of being in a strictly
positive state is roughly 5 times larger than the limiting conditional
probability of being in a strictly negative state.  If the initial state had
been zero, then the two limiting conditional probabilities would have been
equal.  If the process had started in state -10, the third and fourth columns
would swap.

\begin{table}
	\begin{center}
\begin{tabular}{@{}rcccc@{}} \toprule
$n$ & $\frac{K^n(10,0)}{K^n(10,S)}$ &  $\frac{K^n(10,2\NN)}{K^n(10,S)}$
& $\frac{K^n(10,-2\NN)}{K^n(10,S)}$ & $K^n(10,S)$ \\ \midrule
0 & 0.00 & 1.00  & 0.00 &1.00 \\
10 & 0.11 & 0.89 & 0.00 &1.00 \\
20 & 0.38 & 0.60 & 0.02 &0.31 \\
30 & 0.44 & 0.53 & 0.03 &0.042 \\
40 & 0.46 & 0.50 & 0.04 &0.0050 \\
50 & 0.46 & 0.49 & 0.05 &0.00047 \\
$\infty$ & 0.45 & 0.46 & 0.09 & 0.0 \\ \bottomrule
\end{tabular}
\caption{Rates of convergence starting from state 10 with $b = 1/5$ and $2\NN =
	\set{2,4,6, \ldots}$.}\label{Table:from10}
\end{center}
\end{table}
		
\subsubsection{Domain of attraction paradox.}\label{sec:domainofattraction}
The domain of attraction problem  is to determine which initial distributions lead to a particular QSD describing the limiting conditional behavior.
To make things concrete, consider our hub-and-two-spoke example in Fig.~\ref{fig:OurExample}.  
Suppose $b = 1/5$ and $X_0 = 6$, which means that $\pi_6$ describes the limiting conditional
distribution.  
Then, $X_2$ has distribution $K^2(6,\cdot)$, which is 8 with probability $b^2$, 6 with probability $2ab$, and 4 with probability $a^2$.
It might seem obvious that 
$K^2(6, \cdot)$
must be in the domain of attraction of $\pi_6$, but it is not even true.
The limiting conditional behavior is quite
different when these two distributions are used as initial distributions even though
there is no possibility of absorption in 2 steps when starting from state 6.

We will show below that an initial distribution with support on $\set{4,6,8}$ that is in the domain of
attraction of $\pi_6$ is 8 with probability $b^2 (225/28) = 9/28$, 6 with probability $2ab (25/16)=1/2$, and 4 with probability
$a^2 (125/448) = 5/28$.  
This distribution is obtained by letting the mass at $y$ be
\begin{equation}
	\frac{K^{nd}(6, y) \hat{h}(y)}{\sum_z K^{nd}(6,z) \hat{h}(z)},
\end{equation}
where $nd = 2$. Intuition for this choice is given in Remark~\ref{rem:intuition}.

For our hub-and-two-spoke example, we can do much more than locate a few
distribution within the domain of attraction of $\pi_x$; we can give a fairly
complete solution to the domain of attraction problem when the initial
distribution has a finite support.
Initially, assume that the support of $X_0$ is either on the odd
or even integers.  
If the support of $X_0$ is finite
and concentrated on either the evens or odds,  
then the limiting conditional behavior 
of the hub-and-two-spoke model 
is described by $\sigma_\xi$
where $\xi = \E[X_0/(\abs{X_0} + 1)]$.  
By this we mean
\begin{alignat}{2}
 \lim_{n \to \infty} \pr{X_{nd + k} = y \given X_{n d + k} \in S}
 &=  \frac{\sigma_\xi(y)}{\sigma_\xi(S_k)}.   \label{eqn:DoA}
\end{alignat}
To see that \eqref{eqn:DoA} holds, 
Let $\phi(x) = \pr{X_0 = x}$, and  
let $S_0$ denote the class that
includes the support of $\phi$.
Notice that $\pi_x(S_k)$ is a constant for all $x \in S_0$ so that
$\sigma_\xi(S_k) = \pi_x(S_k)$.   
Since we have a periodic Yaglom limit from each state $x$, we know that
\eqref{eqn:PeriodicYaglomLimit0} holds.  Hence, 
\begin{alignat*}{2}
	\lim_{n \to \infty} \sum_x \phi(x) \prs{x}{X_{nd + k} = y \given X_{n d + k} \in S}
	&=  \sum_x \phi(x) \frac{\pi_x(y)}{\pi_x(S_k)}
	\\
	&=  \frac{1}{\pi_x(S_k)}\sum_x \phi(x) {\sigma_{x/(\abs{x} + 1)}(y)}
	\\
	&= \frac{\sigma_\xi(y)}{\sigma_\xi(S_k)} 
\end{alignat*}
where $\xi = \E[X_0/(\abs{X_0} + 1)]$, and we used the finite support to
justify interchanging the limit and sum.

If the initial distribution were 
$K^2(6,\cdot)$ with $b = 1/5$,
then the limiting conditional behavior would be described by $\sigma_\xi$ where
$\xi ={6472/7875}$.
However,
if the initial distribution is 
on 4,6, and 8 with probabilities 5/28, 1/2, and 9/28, resp., 
then  $\E[X_0/(\abs{X_0} + 1)] = 6/7$, and $\sigma_{6/7} = \pi_6$.

Thus, 
we know that all distributions 
with a finite support
concentrated on either the odd or the even integers 
and having 
$\xi = \E[X_0/(\abs{X_0} + 1)]$ 
are in the domain of attraction
of $\sigma_\xi$.  
The case where the initial distribution has a finite
support that includes both even and odd integers 
is now a fairly straightforward mixture of what we have just done.  We leave
the details, starting with the appropriate expression for the r.h.s.\ of 
\eqref{eqn:DoA}, to the reader.
\subsubsection{Harmonic functions arising from ratio limits.}
Suppose $x$, $y$ are both even integers.  From the asymptotic expression for
$K^{2n}$, it follows that
\begin{alignat*}{2}
	\frac{K^{2n}(x,y)}{K^{2n}(0,y)} &\to \frac{h_{\xi}(x)}{h_{\xi}(0)}=h_{\xi}(x)
\end{alignat*}
where $h_\xi$ is the $\rho$-harmonic function given in
\eqref{eqn:harmonic} with ${\xi} = y / (\abs{y} + 1)$.  Similarly, for $x$ odd and $y$ even
integers
\begin{alignat*}{2}
	\frac{K^{2n + 1}(x,y)}{K^{2n}(0,y)} &\to \rho \frac{h_{\xi}(x)}{h_{\xi}(0)}.
\end{alignat*}
We leave the other cases to the reader.
	
The above results are not surprising given the existence of the measure
$\gamma$ described in Prop.~\ref{prop:duality}.
For example if $x$, $y$ are both even integers then
\begin{alignat*}{2}
\frac{K^{2n}(x,y)}{K^{2n}(0,y)}
&=
\frac{\gamma(0)}{\gamma(x)}\frac{K^{2n}(y,x)}{K^{2n}(y,0)}\\
&\to\frac{\gamma(0)}{\gamma(x)}\frac{\pi_y^0(x)}{\pi_y^0(0)}\\
&=\frac{\gamma(0)}{\pi_y(0)}\frac{\pi_y(x)}{\gamma(x)}\\
&=\frac{h_{\xi}(x)}{h_{\xi}(0)}.
\end{alignat*}

\subsubsection{Time reversals and $h$-transforms.}

In addition to the $\rho$-invariant measures for the hub-and-two-spoke example, we have determined all nonnegative $\rho$-harmonic functions for $K$.
Since we have a multitude of $\rho$-invariant measures $\sigma_\xi$ and $\rho$-harmonic functions $h_\xi$, we can define a multitude of time reversals
\begin{alignat}{2}
\ola{K}_\xi(x,y) &= \frac{R\sigma_\xi(y) K(y,x)}{\sigma_\xi(x)}
\label{eqn:olaK}
\intertext{and a multitude of twisted processes (tilted processes, Doob's $h$-transform, Derman--Vere-Jones transform, change-of-measure)}
{\tilde{K}_\xi}(x,y) &= \frac{R K(x,y)h_\xi (y)}{h_\xi (x)} \label{eqn:tildeK}
\end{alignat}
Both $\ola{K}_\xi$ and ${\tilde{K}_\xi}$ are stochastic matrices on $S$ for
every $\xi \in [-1,1]$. However, depending on the choice of $\xi$, the time
reversal's behavior can vary considerably, and similarly for the twisted
process.  In the hub-and-two-spoke example, we have
\begin{alignat*}{2}
\prs{0}{X_1 = 1 \given X_1 \in S} &= \frac{K(0,1)}{ K(0,S)}
= \frac{1}{2},
\end{alignat*}
but, depending on the choice of $\xi$, $\tilde{K}_\xi (0,1)$ and $\ola{K}_\xi (0,1)$ can take any value in the interval $[1/4, 3/4]$.
  Even if the time until absorption is $\zeta > n$ and $n$ is tending to $\infty$, the initial state $x$ still influences the state prior to absorption $X_{\zeta - 1}$, and two steps prior to absorption, $X_{\zeta - 2}$, \dots
%

\subsubsection{Escape probabilities for time reversals and $h$-transforms.}
Let $\overleftarrow{X}^{\xi} = (\ola{X}^{\xi}_0, \ola{X}^{\xi}_1, \ldots)$
denote a Markov chain with transition matrix $\ola{K}_\xi$ as given in \eqref{eqn:olaK}.
The transition matrix $\ola{K}_\xi$ is a birth-death chain on the integers that
is stochastic and transient, so
$\ola{X}^\xi$ must escape to either plus infinity or negative infinity.
We now compute the probability of escaping to plus infinity starting from x;
that is, we compute
$\ola{h}_\xi(x) \coloneqq \pr{\ola{X}^\xi \text{ escapes to } + \infty \given
	\ola{X}^\xi_0 = x}$.

Computing $\ola{h}_\xi(x)$ is easiest in the two extreme cases.  When $\xi =
1$, $\ola{K}_1$ on the negative integers is a \emph{symmetric}, simple random
walk: for $x<0$
\begin{alignat*}{2}
\ola{K}_1(x,x+1)&=\frac{\sigma_{1}(x+1)K(x+1,x)}{\rho \sigma_{1}(x)}\\
&=\frac{((1-\rho)/2a)\left(\sqrt{b/a}\right)^{|x+1|}b}{\rho ((1-\rho)/2a)\left(\sqrt{b/a}\right)^{|x|}}\\
&=\frac{b}{2\sqrt{ab}\sqrt{b/a}} \\
&=1/2, 
\end{alignat*}
 which means that the process cannot escape to negative infinity.  Hence,
$\ola{h}_1(x) = 1$.  Similarly, $\ola{h}_{-1}(x) = 0$.

Now, we can handle the more interesting cases with $-1 < \xi < 1$.  
From the first equation in \eqref{eqn:hittingtwrev} with $\xi^*
= 1$, 
\begin{alignat*}{2}
	\ola{F}_\xi(x, \ell) &= 
	\frac{\sigma_{\xi}(\ell)}{\sigma_{1}(\ell)}
\frac{\sigma_{1}(x)}{\sigma_{\xi}(x)} 
	\ola{F}_{1}(x, \ell) 
\end{alignat*}
Since $\ola{F}_{1}(x, \ell) = 1$ whenever $\ell > x$
and since the walk is nearest neighbor and transient, letting $\ell \to \infty$
gives the desired escape probability
\begin{alignat}{2}
	\ola{h}_\xi(x) &= \frac{1 + \xi}{2}
	\frac{\sigma_{1}(x)}{\sigma_{\xi}(x)}, \label{eqn:hescape}
	\intertext{which can be rewritten as}
	\sigma_{\xi}(x) \ola{h}_\xi(x) &= \frac{1 + \xi}{2}
	{\sigma_{1}(x)}. \label{eqn:hsigident}
\end{alignat}

By considering the other extreme case with $\xi^* = -1$ and letting
$\ell \to -\infty$, we obtain
\begin{alignat}{2}
	\sigma_{\xi}(x) (1-\ola{h}_\xi(x)) &= \frac{1 - \xi}{2}
	{\sigma_{-1}(x)} \label{eqn:hsigidentminus}
\end{alignat}
Adding \eqref{eqn:hsigident} and \eqref{eqn:hsigidentminus} yields
the representation
\begin{alignat}{2}
\sigma_\xi(y) &= \frac{1 + \xi}{2}\sigma_1(y) + \frac{1 - \xi}{2}
\sigma_{-1}(y)
 \qt{ for $\xi \in [0,1]$.}
\label{eqn:repm}
\end{alignat}

%
%
%

Now, we turn our attention to escape probabilities for the twisted
processes.
Let $\tilde{X}^{\xi} = (\tilde{X}^{\xi}_0, \tilde{X}^{\xi}_1, \ldots)$
denote a Markov chain with transition matrix $\tilde{K}_\xi$ as given in
\eqref{eqn:tildeK}.
The transition matrix $\tilde{K}_\xi$ is a birth-death chain on the integers that
is stochastic and transient, so
$\tilde{X}^\xi$ must escape to either plus infinity or negative infinity.
We now compute the probability of escaping to plus infinity starting from x;
that is, we compute
$\tilde{h}_\xi(x) \coloneqq \pr{\tilde{X}^\xi \text{ escapes to } + \infty \given
	\tilde{X}^\xi_0 = x}$.  Recall that $h_\xi(x)$ was defined in
	\eqref{eqn:harmonic}.  

Computing $\tilde{h}_\xi(x)$ is also easiest in the two extreme cases.  When $\xi =
1$, $\tilde{K}^1$ on the negative integers is a \emph{symmetric}, simple random
walk, which means that the process cannot escape to negative infinity.  Hence,
$\tilde{h}_1(x) = 1$.  Similarly, $\tilde{h}_{-1}(x) = 0$.

Now, we can handle the more interesting cases with $-1 < \xi < 1$.  
From the second equation in  \eqref{eqn:hittingtwrev} 
with $\xi^* = 1$, 
\begin{alignat*}{2}
	\tilde{F}_\xi(x, \ell) &= 
	\frac{h_{\xi}(\ell)}{h_{1}(\ell)}
\frac{h_{1}(x)}{h_{\xi}(x)} 
	\tilde{F}_{1}(x, \ell),  
\end{alignat*}
and $\tilde{F}_{1}(x, \ell) = 1$ whenever $\ell > x$.
Since the walk is nearest neighbor and transient, letting $\ell \to \infty$
gives the desired escape probability 
%
%
%
%
\begin{alignat}{2}
	\tilde{h}_{\xi}(x)
&=
\frac{h_1(x)}{h_{\xi}(x)} \frac{1 + \xi}{2} \\
&=
\frac{1 + \abs{x} + x}{1 + \abs{x} + \xi x}\frac{1 + \xi}{2}
\end{alignat}
The analogous result with $\xi^* = -1$ and $\ell \to -\infty$ is that
\begin{alignat}{2}
1 - \tilde{h}_{\xi}(x)
&=
\frac{h_{-1}(x)}{h_{\xi}(x)} \frac{1 - \xi}{2}. 
\end{alignat}
Combining the two gives the representation
\begin{alignat}{2}
	h_{\xi}(x) &= \frac{1 + \xi}{2} h_1(x) + 
	\frac{1 - \xi}{2} 
	h_{-1}(x) 
	\qt{ for $\xi \in [-1,1]$.}
	\label{eqn:newhrep}
\end{alignat}
Although we have determined the escape probabilities for all $h$-transforms,
the escape probabilities when $h = \hat{h}$ defined in \eqref{eqn:lambda} will
play a fundamental role in Yaglom limits in the $R$-transient
case~\cite{Rtransient}.  For our example, $\hat{h} = h_0$ as described in
\autoref{sec:hath}; hence, the probability of escaping to positive infinity
starting from state $x$ for this $h$-transform is
\begin{alignat}{2}
	\tilde{h}_0(x) &=
	\frac{1 + \abs{x} + x}{2(1 + \abs{x})}  
\end{alignat}
In particular, the probability measure $\pi_x$ describing the periodic Yaglom
limit starting from $x$ can be represented as the following convex combination
of two extremal measures
\begin{alignat*}{2}
	\pi_x &=  \tilde{h}_0(x) \, \pi_\infty + (1 -
	\tilde{h}_0(x)) \,  \pi_{-\infty}.
\end{alignat*}
Thus, the escape probabilities corresponding to $\hat{h}$ determine the
proper weights.

\subsubsection{Martin exit and entrance boundaries of the hub-and-two-spoke
	model.}\label{boundary}
This section depends heavily on the works of Dynkin~\cite{Dynkin} and
Woess~\cite{woess,woess-2009}.
Among other things, the Martin exit boundary theory can characterize
all positive harmonic functions, and
the Martin entrance boundary theory, all invariant measures of an irreducible stochastic or
substochastic matrix.  Generally, there seems
to be more interest in exit boundary theory since it is useful in
describing
the limiting behavior of transient processes;
if left unspecified, Martin
boundary theory usually refers to the exit boundary.
Similarly,
the $t$-Martin exit and entrance boundary theory can be used to describe all
$t$-invariant harmonic functions and $t$-invariant measures.  Again, there
seems to be more interest in the exit boundary theory.
Papers that study the
$t$-Martin exit boundary of killed random walks include
Ignatiouk-Robert~\cite{IrinaKilledHalfSpace, IrinaKilledQuadrant}, Doney~\cite{Doney}, Alili and
Doney~\cite{AliliDoney}, Raschel~\cite{Raschel}, and
Lecouvey and Raschel~\cite{LecouveyRaschel}.
Maillard~\cite{Maillard2015} identifies the $t$-invariant measures for
the Bienaym\'{e}--Galton--Watson process $t \geq \rho$. The $\rho$-Martin
entrance boundary for this process is trivial having a single point, and the
corresponding $\rho$-invariant measure is the classic limit of Yaglom.  When the $\rho$-Martin
entrance boundary is trivial, it is impossible to have different (aperiodic or
periodic) Yaglom
limits starting from different initial states.

Fix $t \geq \rho$.  
Although $G_{(x,y)}(t)$ was defined in 
  \eqref{eqn:GenFun}, the Martin boundary definitions will be slightly less
  ugly if we also define
$G_t (x,y) \coloneqq \sum_{n \geq 0} (1/t)^n K^n(x,y)$.
To construct the $t$-Martin boundaries, we  define $\prescript{*}{}M$ and $M^*$,
the $t$-Martin entrance and exit kernels respectively using 0 as the reference
state, as:
\begin{alignat}{2}
	\prescript{*}{}M(x,y) &\coloneqq \frac{G_t (x,y)}{G_t (x,0)}
	\text{, and } M^*(x,y) &\coloneqq \frac{G_t (x,y)}{G_t (0,y)}.
	\label{eqn:Martin-kernel}
\end{alignat}
Recall that $K$ is $R$-transient, and $1/t\leq R$ so $\prescript{*}{}M$ and $M^*$ exist.

Let $\prescript{*}{}{\bar{S}}$ be the smallest compactification such that the $t$-Martin
kernel $\prescript{*}{}M(x,y)$ extends continuously in $x$; i.e., $x_\infty \in
\prescript{*}{}{\bar{S}}$ if
there is a sequence $x_n\in S$ such that $\prescript{*}{}M(x_n,y)$ converges
for every $y$. The limiting measure on $S$ is denoted by  
$\prescript{*}{}M(x_\infty,\cdot)$.
Let $\partial \prescript{*}{}{\bar{S}} \coloneqq \prescript{*}{}{\bar{S}}
\setminus S$ be the boundary of
$\prescript{*}{}{\bar{S}}$.
$\partial \prescript{*}{}{\bar{S}}$ is called the $t$-Martin entrance boundary.
For more details, see of \cite[Chapter~7]{woess-2009}

Similarly
let $\bar{S}^*$ be the smallest compactification such that the $t$-Martin
kernel $M^*(x,y)$ extends continuously in $y$; i.e., $y_\infty \in \bar{S}^*$ if
there is a sequence $y_n\in S$ such that $M^*(x,y_n)$ converges for every $x$.
The limiting function on $S$ is denoted by  $M^*(\cdot, y_\infty )$.
Let $\partial \bar{S}^* \coloneqq \bar{S}^* \setminus S$ be the boundary of $\bar{S}^*$.
$\partial \bar{S}^*$ is called the $t$-Martin exit boundary.

We are particularly interested in the $\rho$-Martin entrance boundary.
Though the $\rho$-Martin entrance boundary for substochastic matrices seems to
have received little attention, it is ideally suited for
 studying Yaglom limits starting from a fixed state since the
 Yaglom limit (periodic or aperiodic) is a
$\rho$-invariant QSD, and the $\rho$-Martin entrance boundary describes all
$\rho$-invariant measures.

For the hub-and-two-spoke example, we will show that the $\rho$-Martin exit and
entrance boundaries both have two points $\set{-\infty, +\infty}$.  
If $x_n \to -\infty$, then $\prescript{*}{}M(x_n, \cdot) \to
\sigma_{-1}(\cdot)/\sigma_{-1}(0)$; if $x_n \to +\infty$, then $\prescript{*}{}M(x_n, \cdot) \to
\sigma_{1}(\cdot)/\sigma_{1}(0)$.  
Thus, we can extend $\prescript{*}{}M$
continuously to the boundary 
$\partial \prescript{*}{}{\bar{S}}=\{-\infty,+\infty\}$
by defining $\prescript{*}{}M(-\infty,
\cdot) \coloneqq \sigma_{-1}(\cdot)/\sigma_{-1}(0)$ and $\prescript{*}{}M(+\infty, \cdot) \coloneqq
\sigma_{1}(\cdot)/\sigma_{1}(0)$.
Since the two limits differ, we cannot extend
$M$ continuously with a smaller compactification.
Similarly, we will  show that if $y_n \to -\infty$, then $M^*(\cdot,y_n) \to h_{-1}(\cdot)
\eqqcolon M^*(\cdot,-\infty)$ and that if $y_n \to +\infty$, then $M^*(\cdot,y_n) \to h_{1}(\cdot)
\eqqcolon M^*(\cdot,+\infty)$, 
which shows that 
the $\rho$-Martin exit boundary is 
$\partial \bar{S}^*=\{-\infty,+\infty\}$.

To verify the claimed limits for the entrance boundary,
let us derive a more convenient expressions for the $\rho$-Martin entrance 
kernel given in \eqref{eqn:Martin-kernel}.  From \eqref{eqn:Martin-kernel} with $t =
\rho$ and using \eqref{eqn:revG},
\begin{alignat}{2}
	\prescript{*}{}M(x,y) &= 
	\frac{\sigma_{\xi^*}(y)}{\sigma_{\xi^*}(0)}  
	\frac{\ola{G}_{\xi^*}(y,x)}{\ola{G}_{\xi^*} (0,x)} 
	\label{eqn:Mtildetransform}
	\\
	&= 
	\frac{\sigma_{\xi^*}(y)}{\sigma_{\xi^*}(0)}  
	\frac{\ola{F}_{\xi^*}(y,x)}{\ola{F}_{\xi^*} (0,x)}
	\frac{\ola{G}_{\xi^*}(x,x)}{\ola{G}_{\xi^*} (x,x)} 
	\\
	&= 
	\frac{\sigma_{\xi^*}(y)}{\sigma_{\xi^*}(0)}  
	\frac{\ola{F}_{\xi^*}(y,x)}{\ola{F}_{\xi^*} (0,x)}.
\end{alignat}
Now, it is simple to compute the claimed limits for $\prescript{*}{}M(x_n,y)$.  If $x_n \to -\infty$, choose
$\xi^* = -1$ so that $\prescript{*}{}M(x_n,y) \to  
	{\sigma_{-1}(y)}/{\sigma_{-1}(0)}$.  
	On the other hand, if $x_n \to +\infty$, choose 
$\xi^* = 1$ so that $\prescript{*}{}M(x_n,y) \to  
	{\sigma_{1}(y)}/{\sigma_{1}(0)}$.

Similarly, the $\rho$-Martin exit kernel can be rewritten using \eqref{eqn:twG} to
obtain 
\begin{alignat}{2}
	M^*(x,y)
	&= 
	\frac{\tilde{F}_{\xi^*}(x,y)}{\tilde{F}_{\xi^*}(0,y)} 
	\frac{h_{\xi^*}(x)}{h_{\xi^*}(0)},\label{eqn:Mstartransform}   
\end{alignat}
which makes it easy to verify the claimed limits for $M^*(x,y_n)$.

For the hub-and-two-spoke model, we have the representation of all
positive $\rho$-invariant measures by \eqref{eqn:repm} and all positive $\rho$-harmonic functions by
\eqref{eqn:newhrep}.
The mapping where state $x \in \ZZ$ is mapped to $x/(1 + \abs{x})$ is a
homeomorphism that connects the boundary points $-\infty$ to $-1$ and $+\infty$
with $1$.
With this homeomorphism, 
\eqref{eqn:newhrep} is the general integral representation of the $\rho$-harmonic
functions
over the $\rho$-Martin exit boundary (see Theorem 6 in \cite{Dynkin})  and \eqref{eqn:repm} will be the integral representation of the
$\rho$-invariant measures over the $\rho$-Martin entrance boundary 
(see Theorem 11 in \cite{Dynkin}).  

From~\cite{vanDoorn1991}, we know that the hub-and-one-spoke model also has a unique
quasi-stationary distribution for every $t \in (\rho,1)$.
We now argue that---similar to the situation when $t = \rho$---the hub-and-two-spoke model has a family of
quasi-stationary distributions for every $t \in (\rho,1)$, which can be
normalized to give a $t$-invariant quasi-stationary distribution.

Fix $t \in (\rho,1)$.  Let $0 < s_1 < s_2 < 1$ be the roots of the $f(r) = ar^2
- tr + b$, which must be real and distinct since the discriminant is positive.
Define
\begin{equation}
	\sigma_-(x) =
	\begin{cases}
		\frac{s_1^x}{2} &\text{ for $x \geq 1$} \\
		C s_1^{\abs{x}} + \frac{s_2}{s_2 - s_1} s_2^{\abs{x}}  &\text{ for $x \leq -1$} \\
		1 &\text{  for $x = 0$}
	\end{cases}
\end{equation}
where $C = (1/2 - s_2/(s_2 - s_1))$.
Even though $C < 0$, $\sigma_-(x)$ is strictly positive since $\sigma_-(x) \geq
s_1^{\abs{x}}/2$ for all $x$.  In addition, $\sigma_-$ is summable since
$\sigma_-(x) \leq s_2^{\abs{x}}$ for all $x$.
Thus, $\sigma_-$ could be normalized to be a proper probability distribution.
Furthermore, $\sigma_-$ is a $t$-invariant measure for $K$.  For $x \geq
2$, the time
reversal of $K$ with respect to $\sigma_-$
gives $\ola{K}(x,x-1) = b/(ts_1)$ and $\ola{K}(x, x+1) = as_1/t$.
For $x \geq 2$, the drift $b/(ts_1) - as_1/t$  is negative if
$s_1 < \sqrt{b/a}$.  But the smaller root $s_1$ is less than  $\sqrt{b/a}$ since
\begin{alignat*}{2}
	f(\sqrt{b/a}) &= (\rho - t) \sqrt{b/a} 
	< 0.
\end{alignat*}
Since the time reversal is $1$-transient, the reversed process converges to
$-\infty$.  Thus, the $t$-Martin entrance boundary will have a point $-\infty$,
and the $t$-Martin entrance kernel at $-\infty$ is $\sigma_-$, which is a
minimal $t$-invariant measure.  By symmetry,
$\sigma_+(x) \coloneqq \sigma_-(-x)$ is a different minimal $t$-invariant
measure corresponding to a
different
$t$-Martin boundary point $+\infty$.  Since a nearest neighbor random walk on
the integers could have at most two points in the (full) Martin
compactification and
we have found two minimal $t$-invariant measures, we have found the whole
Martin compactification for each $t > \rho$.
Any convex combination of the two minimal $t$-invariant measures is also a $t$-invariant
measure.  
Since the $t$-Martin exit boundary does not depend on $t$ for $t \in
[\rho,1)$ and satisfies the conditions given by Woess
\cite[page~301]{woess}, 
the Martin entrance boundary is \emph{strictly stable}.  
From reversibility, the dual $t$-harmonic functions can be determined using the
reversibility measure $\gamma$ \eqref{eqn:h2gamma}.
The Martin exit boundary is also strictly stable.  

\subsection{Time until returning to state zero}\label{subsec:remlife}
We construct an example where the ratio limit $\hat{h}(\cdot )$
that is defined just before Prop.~\ref{periodicproIII}
is not
$\rho$-harmonic; indeed, this example does not have any $\rho$-harmonic
functions.  This example does not contradict Prop.~\ref{periodicproIII}
since
the support of $K(u,\cdot )$ will be infinite when $u = 0$.
Thus, the finite support assumption is needed in Prop.~\ref{periodicproIII}
and,
in general, we cannot limit attention to
$\rho$-harmonic functions when determining $\hat{h}(\cdot)$.  Despite the
above, we compute the limiting conditional distribution.  In addition, if we
consider $K^2$ with state space the even integers, we have an example with a
Yaglom limit that falls outside of Kesten's sufficient
conditions~\cite{KestenSubMarkov} since his condition (1.4u) does not
hold.

This simple example has several other interesting aspects.
The $\rho$-Martin exit
kernel at $+\infty$ is the same as the $\rho$-Martin kernel
at  state 0, which leads to a divergence of approaches on
whether to include $+\infty$ in the boundary. 
Doob \cite{Doob-1959} and Kemeny, Snell and Knapp \cite{K-S-K} would not include
the point $+\infty$, while Hunt \cite{Hunt-1960}, Dynkin \cite{Dynkin}, and Woess
\cite{woess,woess-2009} would include
$+\infty$  allowing  $S$ to  remain discrete in the induced topology;  
see the discussion at \cite[page~189]{woess-2009}.
We follow the latter group.

The example is a variation on the hub-and-one-spoke example.
Returns to state 0 in the hub-and-spoke model form a terminating renewal
process.  Define a sub-Markov chain where--prior to absorption--the state is the
remaining lifetime of
this renewal process; that is, the number of steps until being in state 0.

The state space is $\NN_0$,
and the substochastic transition matrix is (recycling the notation $K$)
    \begin{align}
K &=
\begin{bmatrix}
f_1 & f_2 & f_3 & f_4 & f_5 & \cdots \\
1 & 0 & 0 & 0  & 0 & \cdots \\
0 & 1 & 0 & 0 & 0 & \cdots \\
\vdots
\end{bmatrix}
\end{align}
where $f_n$ is the coefficient of $z^n$ in the generating function $F(z)$ given
in the beginning of \autoref{subsec:SVJRW}.
For this example, $0 = f_1 = f_3 = f_5 = \dots $, $F(R) = 1/2$, $F(1) = b$, and
$K$ is periodic with period 2.

If both the hub-and-one-spoke and the remaining lifetime examples start in
state 0, then
we can couple the two processes so that the
times of visits to state 0 are identical.
It immediately follows that this example is also $R$-transient with $\rho =
2\sqrt{ab}$.
In addition,
an asymptotic expression for $\prs{0}{\zeta > 2n}$ is given
by the r.h.s.\ of \eqref{eqn:alive} with $x = 0$.  This gives an asymptotic
expression for the denominator of $\hat{h}_0(0)$ in
\eqref{eqn:SurvivalRatioPeriodic} where the subscript 0 denotes
the class of even states and $x_0 = 0$ is the reference state.

An asymptotic expression for the numerator can be found by noticing that $\prs{2x}{\zeta > 2n} =
\prs{0}{\zeta > 2n - 2x}$.  Simplifying gives $\hat{h}_0(2x) = R^{2x}$.
Consequently, from the definition of $\hat{h}(\cdot )$ immediately before
Prop.~\ref{periodicproIII}, we have $\hat{h}(x) = R^x$ for $x \in \NN_0$. 
Even though we have computed $\hat{h}(\cdot )$, we now show that $\hat{h}(\cdot
)$ is not $\rho$-harmonic by showing that $K$ does not have any $\rho$-harmonic
functions.

Solving $Kh(x) = \rho h(x)$ for $x \geq 1$ gives $h(x) = R^x
h(0)$, which is looking like $\hat{h}(\cdot )$ above.  However, the equation $Kh(0)
= \rho h(0)$ simplifies to $\rho F(R) h(0) = \rho h(0)$, but $F(R) = 1/2$.
Thus,
there is no non-zero solution, and $K$  does not have any $\rho$-harmonic functions.
The function $\hat{h}(x) = R^x$ is $\rho$-superharmonic.


We can compute the limiting conditional distribution for the remaining lifetime
example as follows.
First, since
\begin{alignat*}{2}
	\frac{K^n(x, y)}{K^n(x,S)} &=
	\frac{K^{n - x}(0, y)}{K^{n - x}(0,S)},
\end{alignat*}
any limiting conditional distribution would not depend on the initial state, so
we can assume that the initial state is 0 without loss of generality.  (This
would also be true if we looked at a two-spoke version of the remaining
lifetime example.)
Consequently, should it exist, let $\pi^k$ to be the periodic Yaglom limit on $S_k$
where $S_0 = 2\NN_0$ and $S_1$ is the odd positive integers analogous to
\eqref{eqn:PeriodicYaglomLimit1} or \eqref{eqn:PeriodicYaglomLimit2}.
Similarly, should it exist, let $\pi$ denote the corresponding \rhoi\qsd.

From the coupling argument, we know that
\begin{alignat*}{2}
	\pi^0(0) &= \lim_{n \to \infty}
	\frac{K^{2n}(0, 0)}{K^{2n}(0,S)} \\
&= \frac{1 - \rho^2}{a}
\intertext{and if a periodic Yaglom limit exists, that}
\pi(S_0) &= 1/(1 + \rho),
\pi(S_1) = \rho/(1 + \rho) \mbox{ and }
\pi(0) = \frac{1 - \rho}{a}.
\end{alignat*}
	
For $y \geq 1$,
\begin{alignat*}{2}
	 K^{2n}(0,2y) &= K^{2n+2y}(0,0) - \sum_{k = 1}^{y } K^{2n+2y
	- 2k}(0,0)f_{2k}
	\\
	&\sim \frac{1}{\sqrt{\pi}} \frac{(4ab)^{n + y}}{n^{3/2}}
	\left[1 - \sum_{k =1}^n f_{2k} R^{2k} \right]
	\\
	\pi^0(2y) &= \rho^{2y} \frac{1 - \rho^2}{a}
	\left[1 - \sum_{k =1}^y f_{2k} R^{2k} \right]
\end{alignat*}
where we used the coupling to see that $K^n(0,0) = P^n(1,1)$,
\eqref{eqn:SVAP} with $x=y=1$, and \eqref{eqn:alive} with $x = 0$.
Since $\sum_{y \geq 1} \rho^{2y} \sum_{1 \leq k \leq y} f_{2k} R^{2k} = b/(1 -
\rho^2)$, it follows that $\pi^0$ is a proper probability measure on $S_0$.
Thus, for every $x \in S_0$, the hypotheses of Prop.~\ref{periodicpro} hold;
hence,
for every $x \in S_0$,
there is a \rhoi\qsd.
But $K$ has a unique
\rhoi\qsd, which is given by
\begin{alignat*}{2}
	\pi(y) &= \rho^y \frac{(1 - \rho)}{a} \left[ 1 - \sum_{k = 1}^y f_k R^k
	\right] \qt{for $y \in \NN_0$,}
\end{alignat*}
(and is quite different from the analogous result for the hub-and-one-spoke
model).
Consequently, the \rhoi\qsd\ $\pi$ must describe the (periodic) limiting conditional
behavior starting from any state $x$.

\subsubsection{Exit and entrance boundaries for the time remaining until
returning to zero}
For the time until returning to zero example, we examine the $\rho$-Martin
exit and entrance
boundaries.  Unlike the hub-and-two-spoke model, the $\rho$-Martin entrance and
exit boundaries
are not the same.  In addition, the minimal $\rho$-Martin exit boundary and 
the $\rho$-Martin exit boundary are not equal.

To compute the $\rho$-Martin kernels, use \eqref{eqn:twG} so that 
\begin{alignat*}{2}
	G_\rho(x,y) 
	&=  
	\tilde{G}(x,y) 
	\frac{\hat{h}(x)}{\hat{h}(y)}
\end{alignat*}
where $\hat{h}(x) = R^x$,
\begin{align}
	    \tilde{K} &=
\begin{bmatrix}
	\tilde{f}_1 & \tilde{f}_2 & \tilde{f}_3 & \tilde{f}_4 & \tilde{f}_5 & \cdots \\
1 & 0 & 0 & 0  & 0 & \cdots \\
0 & 1 & 0 & 0 & 0 & \cdots \\
\vdots
\end{bmatrix}
\end{align}
and $\tilde{f}_k = R f_k \hat{h}(k - 1)/\hat{h}(0)= f_k R^{k}$.
Notice that $\tilde{K}(0,S) = F(R) =  1/2$.
To compute
$\tilde{G}(x,y)$, notice that $\tilde{G}(x,y)=\indicator{x \geq y > 0} + \tilde{G}(0,y) $
where $\indicator{x \geq y > 0}$ is 1 if $x \geq y > 0$ and zero, otherwise.
To compute $\tilde{G}(0,y)$, also notice that the number of upward excursions has a
geometric distribution with parameter $1/2$,  and the expected number of
excursions is 1.  Each excursion visits $y$ with probability
$\tilde{F}_y \coloneqq \tilde{f}_{y + 1} +
\tilde{f}_{y + 2} + \ldots $  Thus, $\tilde{G}(0,y)= \tilde{F}_y $.

Now, the $\rho$-Martin exit kernel is
\begin{alignat*}{2}
	M^*(x,y) &\coloneqq \frac{G_\rho (x,y)}{G_\rho (0,y)} \\
	&=\frac{{\hat{h}(x)}\tilde{G}(x,y) }{{\hat{h}(0)}\tilde{G}(0,y) } \\
	&= R^x\,\frac{ \indicator{x \geq y > 0} +
		\tilde{G}(0,y)}{\tilde{G}(0,y)}
\end{alignat*}
Since $M^*(x,y_n) \to R^x$ as $y_n \to \infty$, there is a single point
$+\infty$ in the $\rho$-Martin exit boundary---even though $M^*(x,+\infty) =
M^*(x,0) = R^x$, a point is still added in the compactification of the state
space; see
\cite[page~189]{woess-2009}. As a consequence, the minimal $\rho$-Martin
exit boundary is the empty set since there are no minimal $\rho$-harmonic
functions; see \cite[page~207]{woess-2009}.   Thus, for this example, 
the $\rho$-Martin exit boundary and 
the minimal $\rho$-Martin exit boundary 
are not
the same.

\subsection{Age example}
Instead of considering the remaining lifetime as in the previous example,
consider the age where the age is the number of steps since being in zero prior
to absorption.
The state space is $\NN_0$,
and the substochastic transition matrix is (recycling the notation $K$ again)
    \begin{align}
K &=
\begin{bmatrix}
0 & b & 0 & 0 & 0 & \cdots \\
r_2 & 0 & 1 - r_2 & 0  & 0 & \cdots \\
r_3 & 0 & 0 & 1 - r_3 & 0 & \cdots \\
\vdots
\end{bmatrix}
\end{align}
where $r_j = f_j/\bar{F}_j$ and $\bar{F}_j = f_j + f_{j + 1} + \dots$.
The age example and the remaining lifetime example can be coupled so that they
are in zero at the same points in time and have the same time of absorption
if both start in state 0.  Thus, the age example is also $R$-transient with
$\rho = 2 \sqrt{ab}$.
This example is also periodic with period 2 since $0 = r_1 = r_3 = \cdots$

The age example seems nicer than the remaining lifetime
example since
the support of $K(u, \cdot)$ is finite for all states $u$.  Nonetheless, $K$
does not possess a \rhoi\qsd.  From Prop.~\ref{periodicpro}, it follows that
\eqref{eqn:PeriodicYaglomLimit1} cannot hold.  Thus, the age example does not
have a (periodic or aperiodic) Yaglom limit.


\subsection{Duality without assuming reversibility}\label{sec:dualworev}
In \autoref{sec:dualwrev}, we described a duality between $t$-invariant
measures and $t$-harmonic functions that relied on reversibility.  We 
describe a generalization that does not rely on reversibility.  

For this duality to hold, we need the following two conditions:
\begin{itemize}
	\item 
		If $\prescript{*}{}{M}(z_n,y)$ converges to a minimal
$t$-invariant measure $\sigma$ for some sequence $z_n$, then 
$M^*(x,z_n)$ also converges to a minimal
$t$-harmonic function $h$. 
	\item 
		If $M^*(x,z_n)$ converges to a minimal
$t$-harmonic function $h$ for some sequence $z_n$, then
$\prescript{*}{}{M}(z_n,y)$ also converges to a minimal
$t$-invariant measure $\sigma$. 
\end{itemize}
A $t$-harmonic function $h$ is minimal if $h(0) = 1$ 
and $h \geq h_1$ where $h_1$ is also
$t$-harmonic implies that $h_1/h$ must be a constant; a minimal $t$-invariant measure
is defined analogously.  
The above two properties imply that 
$\mathcal{M}^* = 
\prescript{*}{}{\mathcal{M}}$
where $\mathcal{M}^*$ is the \emph{minimal} $t$-Martin exit boundary
\cite[page~263--264]{woess} 
and 
$\prescript{*}{}{\mathcal{M}}$
is the \emph{minimal} $t$-Martin entrance boundary.

Let $\sigma$ be any $t$-invariant measure for $K$, and let $h$ be any $t$-harmonic
function.  
The Poisson-Martin integral provides a unique representation of $h$ 
\cite[(24.18)]{woess} 
and
$\sigma$: 
\begin{alignat*}{2}
h(x)
&=
\int_{\mathcal{M}^*} M^*(x,\cdot) d\nu^h 
\\
\sigma(y)
&=
\int_{\prescript{*}{}{\mathcal{M}}}  \prescript{*}{}{M}(\cdot,y) d\nu^\sigma.
\end{alignat*}
When 
$\mathcal{M}^* = 
\prescript{*}{}{\mathcal{M}}$, we can say that $h$ and $\sigma$ are duals of
each other if $\nu^h = \nu^\sigma$ in their Poisson-Martin representations.  

We show that this concept of duality includes the duality described in
\autoref{sec:dualwrev} when $K$ is reversible.
Let $\gamma$ be the reversibility measure of $K$.  
For the first part of the argument, assume that $h$ is a minimal $t$-harmonic
function.  
Hence, there exists a sequence of states $z_n$ such that $M^*(x,z_n) \to
M^*(x,z_\infty) \eqqcolon 
h(x)$, and the unique Poisson-Martin representation has $\nu^h$ being a
point mass. 
We need to show that 
$\prescript{*}{}{\mathcal{M}}(z_n,x) \to 
\sigma(x)\coloneqq h(x) \gamma(x)$ 
and that $\sigma$ is minimal to know that $\sigma$ is the dual of $h$ under the
generalized definition.  

Now
\begin{alignat*}{2}
\prescript{*}{}{{M}}(z_n,x)
&=
\frac{G_t(z_n,x)}{G_t(z_n,0)} \\
&=
\frac{\gamma(x) G_t(x,z_n)/\gamma(z_n)}{\gamma(0)G_t(0,z_n)/\gamma(z_n)} \\
&=
\frac{\gamma(x)}{\gamma(0)} M^*(x,z_n) \\
&\to
{\gamma(x)}{h(x)} = \sigma(x).
\end{alignat*}
In addition, $\sigma$ is minimal.  To see this, assume that $\sigma \geq
\sigma_1$, which is also $t$-invariant.  Hence, $h \geq h_1$ where $h_1(x) =
\sigma_1(x)/\gamma(x)$.  Since $h_1$ is also $t$-harmonic, $h/h_1$ is a constant,
but that means $\sigma/\sigma_1$ is a constant, which means $\sigma$ is
minimal.  
Thus, $\sigma$ must be the dual of $h$ under the generalized definition.

The second part of the argument is analogous to the first part except for
starting off with a minimal $t$-invariant measure $\sigma$ and showing that
$h(x) = \sigma(x)/\gamma(x)$ is the dual of $\sigma$.  

From the first two parts of the argument, it follows that 
$\mathcal{M}^* = 
\prescript{*}{}{\mathcal{M}}$
and that if $z_\infty \in \mathcal{M}^*$ then
$M^*(x,z_\infty) 
=
\prescript{*}{}{M}(z_\infty,x)/\gamma(x)$.  
From the Poisson-Martin integral representation, it now follows that 
that the generalized definition of duality reduces to $h(x) =
\sigma(x)/\gamma(x)$ when $K$ is reversible.  

\section{Literature}\label{sec:Lit}
Van~Doorn and Pollett~\cite{vanDoornPollett} survey the vast
literature on, and the interconnections among, Yaglom limits, limiting
conditional distributions, \qsds\  and $t$-invariant measures.  We discuss only
the most relevant papers, and we consider only
irreducible chains.  Many papers are set in continuous time, but most results for continuous time processes have an obvious analog in discrete time \cite[Section~3.4]{vanDoornPollett}.

In the impressive paper Seneta and Vere-Jones
\cite{Vere-Jones-Seneta}, the authors show that the Yaglom limit does not depend on the
initial state in the $R$-positive case, but they allow for the
possibility that a Yaglom limit might depend on the starting state in the
$R$-transient case.  Seneta and Vere-Jones analyze several examples
including the killed simple random walk on the nonnegative integers that we
exploited and the Bienaym\'{e}-Galton-Watson
process.

That the Yaglom limit might depend on the starting state seems to
have been overlooked for two reasons:  first, in all of the analyzed examples with
Yaglom limits starting from a fixed state, the limit did not depend on the
starting state; second, it
seems like there should be a coupling argument
showing that the limit does not
depend on the starting state.

The connection between the $\rho$-Martin entrance boundary and Yaglom
limits for $\rho \neq 1$ seems to have been largely overlooked except for
Maillard~\cite{Maillard2015},
who studies the minimal $t$-Martin entrance boundary for $t \geq
\rho$ for the subcritical Bienaym\'{e}-Galton-Watson process.  He finds
that the $\rho$-Martin entrance boundary is trivial and that there is a
unique \rhoi QSD, which is the classic Yaglom limit.
Breyer \cite{Breyer} made the profound connection between the space-time Martin
entrance boundary and Yaglom limits but failed to realize that convergence
to the space-time boundary could fail as in Kesten's example \cite{KestenSubMarkov} (see below).
Lalley \cite{Lalley-1991} identifies the space-time Martin boundary of a nearest
neighbor random walk (stochastic transition matrix) on a 
class of homogeneous trees.  

Seneta and Vere-Jones \cite{Vere-Jones-Seneta} mention but do not pursue the
domain of attraction problem.  Most of the domain of attraction work has been
concerned with situations where there is a unique $t$-invariant
distribution for each $t \in [\rho, 1)$, and the question is which initial
distributions are in the domain of attraction of a particular
$t$-invariant distribution; see van Doorn \cite{vanDoorn1991} and
Villemonais \cite{Villemonais2015}.  This is slightly different than
\autoref{sec:domainofattraction} where we have many $\rho$-invariant
distributions, and we determine which initial distributions with a finite
support are in the
domain of attraction of a given \rhoi QSD.

The most amazing example appears in H.~Kesten's \emph{tour de force}
\cite{KestenSubMarkov}.
Certainly, Kesten was aware that the Yaglom limit
could depend on the starting state, but he does not even bother to mention it.
Instead, he focuses on constructing an example
of a sub-Markov chain possessing most every nice property possible including having at least one $\rho$-invariant \qsd, but fails to have a Yaglom limit!
Our far more pedestrian example has a different Yaglom limit for every starting state.
Kesten's example and ours are similar. The only
difference is that Kesten allows  the process to stay in each
state $x$ with probability $r_x$.
Once the process leaves a state, the transition probabilities are
identical to our example.    By carefully choosing the
values of $r_x$, Kesten constructs an example where $\pr{X_n > 0 \given X_0
	= 0, X_n \in S}$ oscillates as $n$ increases and does not converge. 
	Hence, the Yaglom limit fails.
Kesten \cite{KestenSubMarkov} also gives general conditions guaranteeing the existence of a
Yaglom limit that does not depend on the starting state but this result depends on conditions ensuring the
uniqueness of the $\rho$-invariant probability or equivalently the triviality of the $\rho$-Martin
entrance boundary.

\acks
We would like to thank Phil Pollett for his helpful comments, Vadim Kaimanovich
for fielding some Martin boundary questions, and Jaime San
Martin for several technical suggestions.  
We would also
like to thank the referees for their comments, suggestions and questions, 
which we believe substantially improved the paper.

\bibliography{QuasiStationary}

\begin{thebibliography}{10}

\bibitem{AliliDoney}
{\sc Alili, L. and Doney, R.~A.} (2001).
\newblock {M}artin boundaries associated with a killed random walk.
\newblock {\em Ann. Inst. H. Poincar\'e Probab. Statist.\/} {\bf 37,} 313--338.

\bibitem{Billingsley}
{\sc Billingsley, P.} (1995).
\newblock {\em Probability and Measure} third~ed.
\newblock Wiley Series in Probability and Mathematical Statistics. John Wiley
  \& Sons Inc., New York, NY.
\newblock A Wiley-Interscience Publication.

\bibitem{Breyer}
{\sc Breyer, L.~A.} (1998).
\newblock Quasistationarity and {M}artin boundaries: Conditioned processes.
\newblock \url{http://www.lbreyer.com/preprints.html}.

\bibitem{Clark2012}
{\sc Clark, P.~L.} (2012).
\newblock Honors calculus.
\newblock \url{http://math.uga.edu/~pete/2400full.pdf}.

\bibitem{ClarkSeqSer}
{\sc Clark, P.~L.} (2012).
\newblock Sequences and series: A sourcebook.
\newblock \url{http://math.uga.edu/~pete/3100supp.pdf}.

\bibitem{Doney}
{\sc Doney, R.~A.} (1998).
\newblock The {M}artin boundary and ratio limit theorems for killed random
  walks.
\newblock {\em J. London Math. Soc. (2)\/} {\bf 58,} 761--768.

\bibitem{Doob-1959}
{\sc Doob, J.~L.} (1959).
\newblock Discrete potential theory and boundaries.
\newblock {\em J. Math. Mech.\/} {\bf 8,} 433--458; erratum 993.

\bibitem{Dynkin}
{\sc Dynkin, E.~B.} (1969).
\newblock Boundary theory of {Markov} processes (the discrete case).
\newblock {\em Russian Mathematical Surveys\/} {\bf 24,} 1.

\bibitem{FerrariRolla}
{\sc Ferrari, P.~A. and Rolla, L.~T.} (2015).
\newblock Yaglom limit via {H}olley inequality.
\newblock {\em Braz. J. Probab. Stat.\/} {\bf 29,} 413--426.

\bibitem{McFoleyII}
{\sc Foley, R.~D. and McDonald, D.~R.}
\newblock Yaglom limits for ${R}$-recurrent chains.
\newblock manuscript.

\bibitem{Rtransient}
{\sc Foley, R.~D. and McDonald, D.~R.}
\newblock Yaglom limits for ${R}$-transient chains and the space-time {M}artin
  boundary.
\newblock manuscript.

\bibitem{Hunt-1960}
{\sc Hunt, G.~A.} (1960).
\newblock Markoff chains and {M}artin boundaries.
\newblock {\em Illinois J. Math.\/} {\bf 4,} 313--340.

\bibitem{IrinaKilledHalfSpace}
{\sc Ignatiouk-Robert, I.} (2008).
\newblock {M}artin boundary of a killed random walk on a half-space.
\newblock {\em J. Theoret. Probab.\/} {\bf 21,} 35--68.

\bibitem{IrinaKilledQuadrant}
{\sc Ignatiouk-Robert, I. and Loree, C.} (2010).
\newblock {M}artin boundary of a killed random walk on a quadrant.
\newblock {\em Ann. Probab.\/} {\bf 38,} 1106--1142.

\bibitem{JackaRoberts}
{\sc Jacka, S.~D. and Roberts, G.~O.} (1995).
\newblock Weak convergence of conditioned processes on a countable state space.
\newblock {\em J. Appl. Probab.\/} {\bf 32,} 902--916.

\bibitem{Kelly}
{\sc Kelly, F.~P.} (1979).
\newblock {\em Reversibility and stochastic networks}.
\newblock John Wiley \&\ Sons, Ltd., Chichester.
\newblock Wiley Series in Probability and Mathematical Statistics.

\bibitem{K-S-K}
{\sc Kemeny, J.~G., Snell, J.~L. and Knapp, A.~W.} (1976).
\newblock {\em Denumerable {M}arkov chains} second~ed.
\newblock Springer-Verlag, New York-Heidelberg-Berlin.
\newblock With a chapter on Markov random fields, by David Griffeath, Graduate
  Texts in Mathematics, No. 40.

\bibitem{KestenSubMarkov}
{\sc Kesten, H.} (1995).
\newblock A ratio limit theorem for (sub) {Markov} chains on $\{$1,
  2,{\textperiodcentered}{\textperiodcentered}{\textperiodcentered}$\}$ with
  bounded jumps.
\newblock {\em Advances in applied probability\/} 652--691.

\bibitem{Lalley-1991}
{\sc Lalley, S.~P.} (1991).
\newblock Saddle-point approximations and space-time {M}artin boundary for
  nearest-neighbor random walk on a homogeneous tree.
\newblock {\em J. Theoret. Probab.\/} {\bf 4,} 701--723.

\bibitem{LecouveyRaschel}
{\sc Lecouvey, C. and Raschel, K.} (2015).
\newblock $t$-martin boundary of killed random walks in the quadrant.
\newblock {\em ArXiv e-prints\/}.

\bibitem{Maillard2015}
{\sc Maillard, P.} (2015).
\newblock The $λ$-invariant measures of subcritical
  {B}ienaymé--{G}alton--{W}atson processes.
\newblock {\em ArXiv e-prints\/}.

\bibitem{Odlyzko95}
{\sc Odlyzko, A.~M.} (1995).
\newblock Asymptotic enumeration methods.
\newblock In {\em Handbook of combinatorics}. ed. R.~L. Graham, M.~Groetschel,
  and L.~Lovasz.
\newblock vol.~2. Elsevier pp.~1063--1229.
\newblock \url{http://www.dtc.umn.edu/~odlyzko/doc/asymptotic.enum.pdf}.

\bibitem{Pollett-1988}
{\sc Pollett, P.~K.} (1988).
\newblock Reversibility, invariance and {$\mu$}-invariance.
\newblock {\em Adv. in Appl. Probab.\/} {\bf 20,} 600--621.

\bibitem{Pollett-1989}
{\sc Pollett, P.~K.} (1989).
\newblock The generalized {K}olmogorov criterion.
\newblock {\em Stochastic Process. Appl.\/} {\bf 33,} 29--44.

\bibitem{Raschel}
{\sc {Raschel}, K.} (2009).
\newblock {Random Walks in the Quarter Plane Absorbed at the Boundary : Exact
  and Asymptotic}.
\newblock {\em ArXiv e-prints\/}.

\bibitem{SenetaMatrices}
{\sc Seneta, E.} (2006).
\newblock {\em Non-negative matrices and {Markov} chains}.
\newblock Springer Science \& Business Media.

\bibitem{Vere-Jones-Seneta}
{\sc Seneta, E. and Vere-Jones, D.} (1966).
\newblock On quasi-stationary distributions in discrete-time {Markov} chains
  with a denumerable infinity of states.
\newblock {\em Journal of Applied Probability\/} 403--434.

\bibitem{vanDoorn1991}
{\sc van Doorn, E.~A.} (1991).
\newblock Quasi-stationary distributions and convergence to quasi-stationarity
  of birth-death processes.
\newblock {\em Adv. in Appl. Probab.\/} {\bf 23,} 683--700.

\bibitem{vanDoornPollett}
{\sc van Doorn, E.~A. and Pollett, P.~K.} (2013).
\newblock Quasi-stationary distributions for discrete-state models.
\newblock {\em European journal of operational research\/} {\bf 230,} 1--14.

\bibitem{vDSstationarity1995}
{\sc van Doorn, E.~A. and Schrijner, P.} (1995).
\newblock Geomatric ergodicity and quasi-stationarity in discrete-time
  birth-death processes.
\newblock {\em The Journal of the Australian Mathematical Society. Series B.
  Applied Mathematics\/} {\bf 37,} 121--144.

\bibitem{Vere-Jones}
{\sc Vere-Jones, D. et~al.} (1967).
\newblock Ergodic properties of nonnegative matrices i.
\newblock {\em Pacific J. Math\/} {\bf 22,} 361--386.

\bibitem{Villemonais2015}
{\sc Villemonais, D.} (2015).
\newblock Minimal quasi-stationary distribution approximation for a birth and
  death process.
\newblock {\em Electron. J. Probab.\/} {\bf 20,} no. 30, 1--18.

\bibitem{woess}
{\sc Woess, W.} (2000).
\newblock {\em Random Walks on Infinite Graphs and Groups}.
\newblock Cambridge Tracts in Mathematics. Cambridge University Press.

\bibitem{woess-2009}
{\sc Woess, W.} (2009).
\newblock {\em Denumerable {M}arkov chains}.
\newblock EMS Textbooks in Mathematics. European Mathematical Society (EMS),
  Z\"urich.
\newblock Generating functions, boundary theory, random walks on trees.

\end{thebibliography}
\bibliographystyle{apt}
\end{document}